\documentclass[reqno]{amsart}
\usepackage[utf8]{inputenc}
\usepackage{amssymb, amsmath, amsthm, color, enumerate, mathabx, mathrsfs}
\usepackage{bbm}

\usepackage{footmisc} 

\usepackage{pdfpages}

\usepackage{esint} 

\usepackage[normalem]{ulem}

\usepackage{tcolorbox}

\numberwithin{equation}{section}

\newcommand{\R}{\mathbb{R}}
\newcommand{\T}{\mathbb{T}}
\newcommand{\Z}{\mathbb{Z}}
\newcommand{\N}{\mathbb{N}}

\newcommand{\cA}{\mathcal{A}}

\newcommand{\cE}{\mathcal{E}}

\newcommand{\cJ}{\mathcal{J}}
\newcommand{\cL}{\mathcal{L}}

\newcommand{\cN}{\mathcal{N}}
\newcommand{\cQ}{\mathcal{Q}}
\newcommand{\cS}{\mathcal{S}}
\newcommand{\cT}{\mathcal{T}}

\newcommand{\ud}{\mathrm{d}}

\newcommand{\bB}{\mathbf{B}}
\newcommand{\bC}{\mathbf{C}}

\newcommand{\bE}{\mathbf{E}}
\renewcommand{\bf}{\mathbf{f}}
\newcommand{\bg}{\mathbf{g}}

\newcommand{\bw}{\mathbf{w}}

\newcommand{\bL}{\mathbf{L}}

\newcommand{\bp}{\mathbf{p}}
\newcommand{\bq}{\mathbf{q}}

\newcommand{\bT}{\mathbf{T}}

\newcommand{\bzero}{\mathbf{0}}
\newcommand{\bdeta}{\boldsymbol{\eta}}
\newcommand{\bmu}{\boldsymbol{\mu}}
\newcommand{\bxi}{\boldsymbol{\xi}}
\newcommand{\bzeta}{\boldsymbol{\zeta}}
\newcommand{\balpha}{\boldsymbol{\alpha}}

\usepackage{accents}
\newcommand{\dbtilde}[1]{\accentset{\approx}{#1}}

\newcommand{\btmu}{\boldsymbol{\tilde{\mu}}}
\newcommand{\bttmu}{\boldsymbol{\dbtilde{\mu}}}

\newcommand{\sE}{\mathscr{E}}
\newcommand{\sF}{\mathscr{F}}

\newcommand{\fK}{\mathfrak{K}}
\newcommand{\fm}{\mathfrak{m}}
\newcommand{\fN}{\mathfrak{N}}
\newcommand{\fR}{\mathfrak{R}}

\def\inn#1#2{\langle#1,#2\rangle}
\def\BLreg#1#2{\mathrm{BL}_{\mathrm{reg}}(#1,#2)}

\def\BLscale#1#2#3#4{\mathrm{BL}_{#3,#4}(#1,#2)}
\def\BL#1#2{\mathrm{BL}(#1,#2)}

\def\codim#1#2{\mathrm{codim}(#1\,|\,#2)}

\theoremstyle{plain}

\newtheorem{theorem}{Theorem}[section]
\newtheorem{lemma}[theorem]{Lemma}
\newtheorem{proposition}[theorem]{Proposition}
\newtheorem{corollary}[theorem]{Corollary}

\theoremstyle{definition}
\newtheorem{definition}[theorem]{Definition}

\newtheorem{remark}[theorem]{Remark}

\newcommand{\supp}{\mathrm{supp}\,}

\newcommand{\diam}{\mathrm{diam}\,}
\newcommand{\dist}{\mathrm{dist}}

\begin{document}

\title[Localised multilinear restriction]{Localised variants of multilinear restriction}

\date{}

\begin{abstract} We revisit certain localised variants of the Bennett--Carbery--Tao multilinear restriction theorem, recently proved by Bejenaru. We give a new proof of Bejenaru's theorem, relating the estimates to the theory of Kakeya--Brascamp--Lieb inequalities. Moreover, the new proof allows for a substantial generalisation, exploiting the full power of the Kakeya--Brascamp--Lieb theory. 
\end{abstract}

\author[D. Beltran]{ David Beltran }
\address{David Beltran: Departament d’An\`alisi Matem\`atica, Universitat de Val\`encia, Dr. Moliner 50, 46100 Burjassot, Spain}
\email{david.beltran@uv.es}
\thanks{D.B. supported by the AEI grants RYC2020-029151-I  and PID2022-140977NA-I00.}

\author[J. Duncan]{ Jennifer Duncan}
\address{Jennifer Duncan:
Instituto de Ciencias Matem\'aticas CSIC-UAM-UC3M-UCM, C. de Nicolas Cabrera 13-15, 28049 Madrid, Spain}
\email{jennifer.duncan@icmat.es}
\thanks{J.D. supported by ERC grant 834728 and Severo Ochoa grant CEX2019-000904-S}

\author[J. Hickman]{ Jonathan Hickman }
\address{Jonathan Hickman: University of Edinburgh, School of Mathematics, James Clerk Maxwell Building, The King's Buildings, Peter Guthrie Tait Road, Edinburgh, EH9 3FD, UK.}
\email{jonathan.hickman@ed.ac.uk}

\maketitle




\section{Introduction} 




\subsection{Background}\label{subsec: background} For $1 \leq d \leq  n$,  let $\Sigma \colon U \to \R^n$ be a parametrisation of a smooth $d$-dimensional submanifold $S$ of $\R^n$. That is, $U \subset \R^d$ is an open, connected neighbourhood of the origin, $\Sigma \colon U \to \R^n$ is a smooth, injective, regular map\footnote{In particular, $\Sigma \in C^{\infty}(U)$ and $\bigwedge_{j = 1}^{d} \tfrac{\partial \Sigma}{\partial u_j} (u) \neq 0$ for all $u \in U$.} and $S = \Sigma(U)$. By an abuse of notation, we shall often simply refer to $S$, with the tacit understanding that our analysis depends on a choice of parametrisation. 

Given $a \in C^{\infty}_c(U)$ satisfying $0 \leq a(u) \leq 1$ for all $u \in U$, we define the \textit{extension operator} $E_S$ associated to $S$ (or, more precisely, $\Sigma \colon U \to \R^n$) by
\begin{equation}\label{eq: extension def}
    E_S f(x) := \int_U e^{i x \cdot \Sigma(u)} f(u)\, a(u)\ud u, \qquad f \in L^1(S), \,\, x \in \R^n.
\end{equation}
Here and below, again by an abuse of notation, we write $L^p(S)$ in place of $L^p(U)$ for $1 \leq p \leq \infty$. We refer to $a$ as the \textit{amplitude} of $E_S$ and let $\rho(E_S) := \diam \supp a > 0$.

For $\xi = \Sigma(u) \in S$, where $u \in U$, the \textit{tangent space} $T_{\xi}S$ is the $d$-dimensional subspace of $\R^n$ spanned by the vectors $\tfrac{\partial \Sigma}{\partial u_1}(u), \dots, \tfrac{\partial \Sigma}{\partial u_d}(u)$, and the \textit{normal space} $N_{\xi}S$ is the orthogonal complement of $T_{\xi}S$ in $\R^n$.

\begin{definition} Fix $2 \leq k \leq n$ and for $1 \leq j \leq k$ let $S  = \Sigma_j(U_j)$ be a smooth hypersurface in $\R^n$. Further, let $\bq_k := (q_k)_{j=1}^k$ where $q_k := \tfrac{2}{k-1}$. We say $\sE = ((S_j)_{j=1}^k, \bq_k)$ is a \textit{transverse ensemble} if there exists some $\nu > 0$ such that\footnote{For $1 \leq j \leq k$, let $V_j$ be a vector subspace of $\R^n$ of dimension $d_j$ and $\{v_{j, 1}, \dots, v_{j, d_j}\}$ be a choice of orthonormal basis of $V_j$. We define 
\begin{equation*}
   |V_1 \wedge \cdots \wedge V_k| := \Big| \bigwedge_{j = 1}^k \bigwedge_{k = 1}^{d_j} v_{j,k}\Big|,
\end{equation*}
noting that this quantity is independent of the choice of orthonormal bases.}
\begin{equation*}
    \Big| \bigwedge_{j = 1}^k N_{\xi_j} S_j\Big| \geq \nu > 0 \qquad \textrm{where $\xi_j := \Sigma_j(0) \in S_j$ for $1 \leq j \leq k$.}
\end{equation*}
\end{definition}

A fundamental result in modern Fourier analysis is the celebrated Bennett--Carbery--Tao theorem. Here and below, $Q_R := [-R, R]^n$ for $R > 0$.

\begin{theorem}[Bennett--Carbery--Tao {\cite[\S5]{BCT2006}}]\label{thm: BCT}  Let $2 \leq k \leq n$ and $\sE = ((S_j)_{j = 1}^k, \bq_k)$ be a transverse ensemble in $\R^n$. There exists $\rho_{\sE} > 0$ such that for $\max_j\rho(E_{S_j}) < \rho_{\sE}$ the following holds. For all $\varepsilon >0$, there exists a constant $C_{\varepsilon} \geq 1$ such that
\begin{equation}\label{eq: BCT}
\int_{Q_R}\prod_{j = 1}^k |E_{S_j} f_j|^{q_k} \leq C_{\varepsilon} R^{\varepsilon} \prod_{j = 1}^k \|f_j\|_{L^2(S_j)}^{q_k}
\end{equation}
holds for all $R \geq 1$ and $f_j \in L^2(S_j)$, $1 \leq j \leq k$.  
\end{theorem}

Here we investigate localised variants of Theorem~\ref{thm: BCT} where some of the functions $f_j$ are assumed to be supported in thin sets. Under such hypotheses, we can hope to gain in the size of the constant appearing on the right-hand side of \eqref{eq: BCT}. This phenomenon was observed in recent papers of Bejenaru~\cite{Bejenaru2017, Bejenaru2022}, and we begin by describing the existing results. 

Let $S = \Sigma(U)$ be a $d$-dimensional submanifold of $\R^n$, as above. For $1 \leq d' \leq d$, we say $S' \subseteq S$ is a \textit{$d'$-dimensional submanifold of $S$} (or a \textit{codimension $d-d'$ submanifold of $S$}) if $S' = \Sigma(M)$ for $M = \gamma(U')$ a $d'$-dimensional submanifold of $\R^d$ with parametrisation $\gamma \colon U' \to U$, where we always assume that $\gamma(U')$ is compactly contained in $U$ and $\gamma(0) = 0$. In this case, we write $\codim{S'}{S} := d - d'$. Note that $S'$ is also a $d'$-dimensional submanifold of $\R^n$, with $S' = \Sigma'(U')$ for $\Sigma' := \Sigma \circ \gamma$ satisfying 
\begin{equation}\label{eq: compatible param}
    \Sigma(0) = \Sigma'(0). 
\end{equation}
Continuing to abuse notation, given $\mu > 0$ we say $f \in L^2(S)$ is \textit{supported in $\cN_{\mu} S'$}, or $\supp f \subseteq \cN_{\mu}S'$, if $f$ is essentially supported in $\cN_{\mu} M$, where $\cN_{\mu} M \subseteq \R^d$ is the (open Euclidean) $\mu$-neighbourhood of $M$.

\begin{definition} Fix $2 \leq k \leq n$ and for $1 \leq j \leq k$ let $S_j$ be a smooth hypersurface and $S_j' = \Sigma_j'(U_j')$ be a smooth submanifold of $S_j$. Further let $\bq_k := (q_k)_{j=1}^k$ where $q_k := \tfrac{2}{k-1}$. We say $\sE = ((S_j, S_j')_{j = 1}^k, \bq_k)$ is a \textit{transverse ensemble} in $\R^n$ if for some $\nu >0$ we have
\begin{equation}\label{eq: trans}
     \Big|\bigwedge_{j = 1}^k N_{\xi_j'}S_j'\Big| \geq \nu  \qquad \textrm{where $\xi_j' := \Sigma_j'(0) \in S_j'$ for $1 \leq j \leq k$.}
\end{equation}
We say $\sE$ has \textit{codimension} $(m_j)_{j = 1}^k$, where $m_j := \codim{S_j'}{S_j}$ for $1 \leq j \leq k$.
\end{definition}

With the above definition, the local version of Theorem~\ref{thm: BCT} reads as follows. 

\begin{theorem}[Bejenaru \cite{Bejenaru2022}]\label{thm: Bejenaru}
    Let $2 \leq k \leq n$ and $\sE = ((S_j, S_j')_{j = 1}^k, \bq_k)$ be a transverse ensemble in $\R^n$ of codimension $(m_j)_{j=1}^k$. There exists a constant $\rho_{\sE} > 0$ such that for $\max_j\rho(E_{S_j}) < \rho_{\sE}$ the following holds. For all $\varepsilon >0$, there exists a constant $C_{\varepsilon} \geq 1$ such that
\begin{equation}\label{eq: local multlinear}
\int_{Q_R}\prod_{j=1}^k|E_{S_j}f_j|^{q_k} \leq C_{\varepsilon} R^{\varepsilon} \prod_{j=1}^k\mu_j^{m_jq_k/2} \prod_{j=1}^k\|f_j\|_{L^2(S_j)}^{q_k},
    \end{equation}
    holds for all $R \geq 1$ and $f_j \in L^2(S_j)$ satisfying $\supp f_j \subseteq \cN_{\mu_j}S_j'$ for some $0 < \mu_j < 1$. 
\end{theorem}

An earlier (and more restricted) version of this result appeared in \cite{Bejenaru2017}. We remark that it is possible to take $m_j = 0$ in Theorem~\ref{thm: Bejenaru}; in particular, if $m_1 = \dots = m_k = 0$, then we recover Theorem~\ref{thm: BCT}. Note that in this case the support condition is vacuous.

Theorem~\ref{thm: Bejenaru} has recently found two distinct applications:
\begin{itemize}
    \item In \cite{Bejenaru2022a}, Theorem~\ref{thm: Bejenaru} is combined with a variant of Guth's polynomial partitioning method \cite{Guth2016} to establish a sharp $(n-1)$-linear restriction estimate under a curvature hypothesis (see also \cite{Oh2023}).
    \item More recently, in \cite{BDH}, a special case of Theorem~\ref{thm: Bejenaru} (under a low regularity hypothesis) was used to prove $L^p \to L^q$ estimates for the maximal function associated to dilates of a helix in $\R^3$.
\end{itemize}  

One aim of this note is to give a new proof of Theorem~\ref{thm: Bejenaru}. For this, we follow the standard induction-on-scale framework of \cite{BBFL2018, BCT2006} which perhaps provides a more contextualised approach than that of \cite{Bejenaru2017, Bejenaru2022}. In particular, we clarify the relationship between \eqref{eq: local multlinear} and Kakeya-type inequalities. In \cite[p.1588]{Bejenaru2017} the author raises the question of whether \eqref{eq: local multlinear} has a multilinear Kakeya counterpart. Here we show that the corresponding geometric estimates are Kakeya--Brascamp--Lieb inequalities of the type studied in, for instance, \cite{BBFL2018, Zhang2018} and their discretised/regularised variants, as studied in \cite{Maldague2022, Zorin-Kranich2020}. Moreover, the new approach leads to our main result (Theorem~\ref{thm: nested}), which is a substantial generalisation of Theorem~\ref{thm: Bejenaru}, making full use of the Kakeya--Brascamp--Lieb theory. 




\subsection{Regularised Brascamp--Lieb inequalities}\label{subsec: reg BL} In order to state the main result, we briefly recall some elements of Brascamp--Lieb theory. For $2 \leq k \leq n$ and $1 \leq n_j \leq n$ for $1 \leq j \leq k$, a \textit{Brascamp--Lieb datum} is a pair $(\bL, \bp) := ((L_j)_{j=1}^k, (p_j)_{j=1}^k)$ where $L_j \colon \R^n \to \R^{n_j}$ are linear surjective maps and $p_j \in (0,1]$. We let $\BLreg{\bL}{\bp}$ denote the best constant $C > 0$ for which the inequality
    \begin{align}\label{eq: BLreg def}
        \int_{\R^n}\prod_{j=1}^k (f_j\circ L_j)^{p_j} \leq C\prod_{j = 1}^k\Big(\int_{\R^{n_j}}f_j \Big)^{p_j}
    \end{align}
holds for all non-negative functions $f_j \in L^1(\R^{n_j})$ that are constant on cubes in the unit cube lattice $\cQ^{n_j}:=[0,1)^{n_j}+\Z^{n_j}$. The inequalities \eqref{eq: BLreg def} are typically referred to as \textit{regularised} Brascamp--Lieb inequalities and were first introduced in \cite{BCCT2010}. Moreover, in \cite{BCCT2010} it was shown that\footnote{Here the notation $V \leq \R^n$ denotes that $V$ is a linear subspace of $\R^n$.}
\begin{equation}\label{eq:BLreg finite}
    \BLreg{\bL}{\bp} < \infty  \quad \Longleftrightarrow \quad  \dim(V) \leq \sum_{j = 1}^k p_j \dim(L_j V) \quad \text{for all $V \leq \R^n$.}
\end{equation}
 An advantage of regularised Brascamp--Lieb inequalities is that the characterisation \eqref{eq:BLreg finite} for the finiteness of $\BLreg{\bL}{\bp}$ does not involve the scaling condition $\sum_{j = 1}^k p_j n_j = n$, which is necessary in the standard Brascamp--Lieb theory~\cite{BCCT2008}. 




\subsection{Main result}\label{subsec: main result} Our main theorem generalises Theorem~\ref{thm: Bejenaru} in three directions. First, rather than just considering hypersurfaces $S_j$, we allow the $S_j$ to be of arbitrary dimension. Secondly, rather than considering a single submanifold $S_j' \subseteq S_j$  for each $1 \leq j \leq k$, we consider nested families. The nested setup naturally arises in recent polynomial partitioning approaches to multilinear restriction and Kakeya: see \cite{ GWZ, HRZ2022, HZ, Zahl2021}. Finally, rather than work with the transversality condition \eqref{eq: trans}, we formulate our transversality hypothesis in terms of regularised Brascamp--Lieb constants, as defined above. This Brascamp--Lieb formulation greatly relaxes the constraints on the permissible codimensions. 

We first explain the nested setup. Suppose $S_1$ is a $d_1$-dimensional submanifold of a $d_0$-dimensional submanifold $S_0 = \Sigma_0(U_0)$ of $\R^n$, as in \S\ref{subsec: background}. Recall that $S_1$ is also a submanifold of $\R^n$ with parametrisation $\Sigma_1 := \Sigma_0 \circ \gamma_1$ for some map $\gamma_1 \colon U_1 \to U_0$, where $U_1 \subseteq \R^{d_1}$. It therefore makes sense to  consider a $d_2$-dimensional submanifold $S_2$ of $S_1$, which itself is a submanifold of $\R^n$ with parameterisation $\Sigma_2 := \Sigma_1 \circ \gamma_2 = \Sigma_0 \circ \gamma_1 \circ \gamma_2$ for some map $\gamma_2 \colon U_2 \to U_1$, where $U_2 \subseteq \R^{d_2}$. Continuing in this way, we can construct a whole nested sequence of submanifolds. 

To describe the setup precisely,
we say $\cS = (S_0, (S_{\ell})_{\ell = 1}^r)$ is a \textit{nested family of submanifolds of $\R^n$} if $S_0 = \Sigma_0(U_0) \subseteq \R^n$ is a submanifold of $\R^n$ and $S_{\ell}$ is a submanifold of $S_{\ell - 1}$ for $1 \leq \ell \leq r$. Thus, if $d_\ell=\dim(S_\ell)$ for $1 \leq \ell \leq r$, then we have a family of open neighbourhoods $U_{\ell} \subseteq \R^{d_{\ell}}$ of the origin and parametrisations $\gamma_{\ell} \colon U_{\ell} \to U_{\ell-1}$ with $\gamma_{\ell}(U_{\ell})$ compactly contained in $U_{\ell-1}$ and $\gamma_{\ell}(0) = 0$ such that
\begin{equation}\label{eq: chain of param}
    S_{\ell} = \Sigma_{\ell}(U_{\ell}) \quad \textrm{where} \quad \Sigma_{\ell} := \Sigma_0 \circ \sigma_{\ell} \quad \textrm{for} \quad \sigma_{\ell} := \gamma_1 \circ \cdots \circ \gamma_{\ell}.
\end{equation}
In this case, each set $M_\ell:=\sigma_\ell(U_\ell)$ is a $d_\ell$-dimensional submanifold of $\R^{d_0}$ and $S_\ell=\Sigma_0(M_\ell)$. Given $\bmu = (\mu_{\ell})_{\ell = 1}^r \subset (0,1]$, we shall say $f \in L^2(S)$ is \textit{supported in $\cN_{\bmu} \cS$}, or $\supp f \subseteq \cN_{\bmu}\cS$, if $\supp f \subseteq \cN_{\mu_{\ell}} S_{\ell}$ for all $1 \leq \ell \leq r$, in the sense defined above: that is, $\supp f \subseteq \cN_{\mu_\ell} M_\ell$, where $\cN_{\mu_{\ell}} M_\ell \subseteq \R^d$ is the (open Euclidean) $\mu_\ell$-neighbourhood of $M_\ell$. We also assume that $d_r < \cdots < d_1 < d_0$ since this is the configuration of interest for our main theorem; in particular, $r \leq n$.

\begin{definition}\label{def: nested} Fix $2 \leq k \leq n$ and suppose $\sE = ((\cS_j)_{j=1}^k, \bq)$ where the $\cS_j = (S_{j,0}, (S_{j,\ell})_{\ell = 1}^{r_j})$ are nested families of submanifolds and $\bq = (q_j)_{j=1}^k \subset (0,2]$ is a sequence of exponents. Define the Brascamp--Lieb datum
\begin{equation}\label{eq: map ensemble def}
  (\bL(\sE),\bp) = ((L_j(\sE))_{j=1}^k, (p_j)_{j=1}^k)  
\end{equation}
by setting
\begin{equation*}
  L_j(\sE) := \partial_s \Sigma_{j,r_j}(0)^{\top} \colon \R^n \to \R^{d_{j,r_j}} \quad \textrm{and} \quad p_j := q_j/2 \quad \textrm{for $1 \leq j \leq k$,}
\end{equation*}
where $d_{j,r_j} := \dim(S_{j,r_j})$ and $S_{j,r_j} = \Sigma_{j,r_j}(U_{j,r_j})$, with $\Sigma_{j,r_j}$ parametrisations constructed as in \eqref{eq: chain of param}.

\begin{enumerate}[a)]
    \item We say $\sE$ is a \textit{transverse ensemble} in $\R^n$ if 
\begin{equation*}
    \BLreg{\bL(\sE)}{\bp} < \infty.
\end{equation*}
    \item We say $\sE$ has \textit{codimension} $ (m(j,\ell))_{j, \ell}$, where $m(j,\ell) := \codim{S_{j,\ell}}{S_{j,\ell-1}}$ for $1 \leq \ell \leq r_j$ and $1 \leq j \leq k$.
    \item Finally, we say $(\bmu_j)_{j=1}^k$ where $\bmu_j =  (\mu_{j,\ell})_{\ell = 1}^{r_j} \subset (0,1]$ is a \textit{compatible family of scales} for $\sE$ if $0 < \mu_{j,1} \leq \dots \leq \mu_{j, r_j} < 1$ for $1 \leq j \leq k$.
\end{enumerate} 
\end{definition}

With the above definitions, the main theorem reads as follows.

\begin{theorem}\label{thm: nested} Let $2 \leq k \leq n$ and $\sE = ((\cS_j)_{j=1}^k, \bq)$ be a transverse ensemble in $\R^n$ of codimension $(m(j,\ell))_{j, \ell}$, with $\cS_j = (S_j, (S_{j,\ell})_{\ell = 1}^{r_j})$. There exists $\rho_{\sE} > 0$ such that for $\max_j\rho(E_{S_j}) < \rho_{\sE}$ the following holds. For all $\varepsilon >0$, there exists a constant $C_{\varepsilon} \geq 1$ such that
\begin{equation*}
\int_{Q_R}\prod_{j=1}^k|E_{S_j}f_j|^{q_j} \leq C_{\varepsilon} R^{\varepsilon} \prod_{j=1}^k\prod_{\ell = 1}^{r_j}\mu_{j,\ell}^{m(j,\ell)q_j/2} \prod_{j=1}^k\|f_j\|_{L^2(S_j)}^{q_j},
    \end{equation*}
    for all $R \geq 1$ and all $f_j \in L^2(S_j)$ satisfying $\supp f_j \subseteq \cN_{\bmu_j}\cS_j$ for $(\bmu_j)_{j=1}^k$ a compatible family of scales for $\mathscr{E}$.
\end{theorem}

Recall that we are tacitly working with a fixed choice of parametrisations for the $S_j$ and $S_{j,\ell}$, and the constant $C_{\varepsilon}$ will depend (amongst other things) on this choice. If $\Sigma_j \colon U_j \to \R^n$ and $\Sigma_{j,r_j} \colon U_{j,r_j} \to \R^n$ are our fixed parametrisations for $S_j$ and $S_{j,r_j}$ as submanifolds of $\R^n$ then, as discussed in \eqref{eq: chain of param} and \eqref{eq: compatible param}, we always assume $\Sigma_{j,r_j} := \Sigma_j \circ \gamma_{j,1} \circ \cdots \circ \gamma_{j,r_j}$ for intermediate parametrisations $\gamma_{j,\ell}$, and $\Sigma_{j,r_j}(0) = \Sigma_{j}(0)$.

\begin{remark}\label{rmks: main} To contextualise the result, we make the following remarks.\medskip

\noindent a) Let $2 \leq k < n$. A prototypical application of our theorem is a $k$-linear estimate for the $n$-dimensional paraboloid in which one of the functions is localised to an anisotropic box. More precisely, let $U :=(-1,1)^{n-1}$  and $S:=\Sigma(U)$, where $\Sigma(u) :=(u,|u|^2)$. Let $0<\delta \ll 1$ and $B_\delta \subseteq \R^{n-1}$ be the $1 \times \cdots \times 1 \times  \delta \times \delta^2 \times \cdots \times \delta^{n-k}$ box centered at the origin pointing along the directions of the standard basis in $\R^{n-1}$. Let $u_1, \dots, u_{k-1}\in U$ satisfy 
\begin{equation}\label{eq: transv hyp example}
\Big|\bigwedge_{j=1}^{k-1}\frac{(2u_j,-1)}{|(2u_j,-1)|} \wedge  e_{k} \wedge \cdots \wedge e_{n}\Big| \geq 1/2.    
\end{equation}
If $U_1, \dots U_{k-1}$ are sufficiently small neighbourhoods of the points $u_1, \dots, u_{k-1}$  and $U_k$ is a sufficiently small neighbourhood of the origin, then for all $\varepsilon >0$ there exists a constant $C_{\varepsilon} \geq 1$ such that 
\begin{equation}\label{eq: example}
\int_{Q_R}\prod_{j=1}^k|E_{S}f_j|^{\frac{2}{k-1}} \leq C_{\varepsilon} R^{\varepsilon}  \delta^{\frac{(n-k)(n-k+1)}{2(k-1)}} \prod_{j=1}^{k}\|f_j\|_{L^2(U_j)}^{\frac{2}{k-1}}
\end{equation}
for all $R \geq 1$ and all $f_j \in L^2(U_j)$ for $1 \leq j \leq k$ with $\supp f_k \subseteq B_\delta$.

To see how this follows from Theorem \ref{thm: nested}, define $U_{k, \ell} :=(-1,1)^{n-1-\ell}$ and $\gamma_{k,\ell}: U_{k, \ell} \to U_{k, \ell-1}$ by $\gamma_{k,\ell}(s) :=(s,0)$ for $1 \leq \ell \leq n-k$, with $U_{k,0}:=U$. It is then clear that $\supp f_k \subseteq \cN_{\bmu_k} \cS_k$, where $\bmu_k = (\delta^{n-k+1-\ell})_{\ell=1}^{n-k}$ and $\cS_k=(S, (S_{k,\ell})_{\ell=1}^{n-k})$ for $S_{k,\ell}=\Sigma_{k,\ell}(U_{k,\ell})$ and $\Sigma_{k,\ell}:=\Sigma \circ \gamma_{k,1} \circ \cdots \circ \gamma_{k,\ell}$. For all $1 \leq j \leq k-1$, one can centre $U_j$ at the origin by a change of variables, giving rise to a neighbourhood $U_j^*$ and an extension operator $E_{S_j}$ with $S_j=\Sigma_j(U_j^*)$ and $\Sigma_j(u)=(u+u_j, |u+u_j|^2)$. Considering the ensemble $\mathscr{E}=((\cS_j)_{j=1}^{k},\bq_k)$, where $\cS_j=(S_j)$ for $1 \leq j \leq k-1$ and $\bq_k=(q_k)_{j=1}^k$ for $q_k=\frac{2}{k-1}$, we have that  $\ker L_j(\mathscr{E})=(2 u_j,-1)^\top$ for all $1 \leq j \leq k-1$ and $\ker L_k(\mathscr{E})=\mathrm{span}\{e_k, \dots, e_n\}$. One can show that \eqref{eq: transv hyp example} implies that the condition \eqref{eq:BLreg finite} holds with $p_j=\frac{1}{k-1}$ (see, for instance, Corollary \ref{cor: BL to Bejenaru}), and thus $\mathscr{E}$ is a transverse ensemble. Consequently, the claimed multilinear inequality \eqref{eq: example} follows from Theorem \ref{thm: nested}.

\medskip

\noindent b) Taking $\cS_j = (S_j, S_j')$  for $S_j \subset \R^n$ a hypersurface and $S_j'$ a submanifold of $S_j$, for $1 \leq j \leq k$, and $\bp_k = (p_k)_{j=1}^k$ for $p_k = \tfrac{1}{k-1}$, we recover the localised multilinear estimates of Theorem~\ref{thm: Bejenaru}. Indeed, this follows from the observation
\begin{equation*}
  \Big|\bigwedge_{j = 1}^k N_{\xi_j'}S_j'\Big| > 0 \quad \iff \quad   \BLreg{\bL(\mathscr{E})}{\bp_k} < \infty,
\end{equation*}    
where the $\xi_j'$ are as defined in \eqref{eq: trans}: see \S\ref{sec:KBL} for details.\medskip

\noindent c) Suppose $r_j = 0$ for all $1 \leq j \leq k$, so that $\cS_j = (S_j)$ and there is no nesting. If we further suppose that the scaling condition $\sum_{j=1}^k p_j \dim(S_j) = n$ is satisfied, then Theorem~\ref{thm: nested} amounts to the Restriction--Brascamp--Lieb estimates of \cite[Theorem 1.3]{BBFL2018}. If we drop the scaling condition, then the corresponding estimates are not documented in the literature, but can be deduced from the Kakeya--Brascamp--Lieb estimates in \cite{Maldague2022, Zorin-Kranich2020} (see Theorem~\ref{thm: Mald} below) by the argument in \cite{BBFL2018, BCT2006}.\medskip

\noindent d) As discussed in \S\ref{sec: BL ensemble}, the hypothesis $\BLreg{\bL(\sE)}{\bp} <\infty$ of Theorem~\ref{thm: nested} implies 
\begin{equation}\label{eq: BBFL condition}
   \BLreg{(\partial_s \Sigma_j (0)^\top)_{j=1}^k}{\bp}<\infty.
\end{equation}
By the argument of \cite{BBFL2018} (\textit{cf.} item c) above), the condition \eqref{eq: BBFL condition} implies the multilinear estimates
\begin{equation*}
\int_{Q_R}\prod_{j=1}^k|E_{S_j}f_j|^{q_j} \leq C_{\varepsilon} R^{\varepsilon} \prod_{j=1}^k\|f_j\|_{L^2(S_j)}^{q_j}.
\end{equation*}
The relevance of our theorem is therefore the improvement given by the factor $\prod_{j=1}^k \prod_{\ell=1}^{r_j} \mu_{j,\ell}^{m(j,\ell)q_j/2}$ under the hypotheses $\supp f_j \subseteq \cN_{\bmu_j} \cS_j$.
\end{remark}




\subsection{Proof sketch} We describe the main ideas in the simpler setting of Theorem~\ref{thm: Bejenaru}. We remark that, due to the local nature of the desired estimates, we may assume without loss of generality that $\max_j \mu_j < \mu_{\sE}$ for some $\mu_{\sE} > 0$ depending only on $\sE$. Under the hypotheses of Theorem~\ref{thm: Bejenaru}, it is possible to bound
\begin{equation}\label{eq: slice bound}
    |E_{S_j} f_j(x)| \leq \int_{[-\mu_j, \mu_j]^m} |E_{S_j'(\bdeta_j)} f_{j, \bdeta_j}(x)|\,\ud \bdeta_j
\end{equation}
where each $E_{S_j'(\bdeta_j)}$ is an extension operator associated to a submanifold $S_j'(\bdeta_j)$ of $S_j$ of codimension $m_j$: see \S\ref{subsec: slice formula} for details. Moreover, \eqref{eq: trans} and continuity imply that, for $\mu_j$ sufficiently small, the submanifolds $(S_j'(\bdeta_j))_{\ell = 1}^k$ satisfy certain transversality conditions uniformly over all choices of $\bdeta_j \in [-\mu_j, \mu_j]^{m_j}$. For each index $j$, we plug the bound \eqref{eq: slice bound} into the left-hand side of \eqref{eq: local multlinear}. If one were then able to switch the order of the $L^{q_k}_x$-norm and the $L^1_{\bdeta_j}$-norms, for instance using Minkowski's integral inequality, then the multilinear restriction estimates of \cite{BBFL2018} (or certain regularised variants) would apply directly to give, together with an application of the Cauchy--Schwarz inequality, Theorem~\ref{thm: Bejenaru}. However, Minkowski's integral inequality requires $q_k \geq 1$, which is only valid for $k = 2$ or $k = 3$. See \cite[Appendix A]{BDH} for an example of this strategy in the $k=3$ case. 

A similar issue with the exponent $q_k < 1$ arises when passing between the equivalent `neighbourhood' and `extension' formulations of the Bennett--Carbery--Tao restriction theorem in \cite{BCT2006}. One strategy to deal with this is to use a vector-valued extension of the multilinear restriction estimates, derived from the Marcinkiewicz--Zygmund theorem: see \cite[Appendix]{Tao2020} for details. The present situation is somewhat more complicated than that in \cite{BCT2006, Tao2020}, since the operators $E_{S_j'(\bdeta_j)}$ typically have a non-trivial dependence on $\bdeta_j$ (in the context of \cite{BCT2006, Tao2020}, the analogous extension operators are independent of $\bdeta_j$). This means that the Marcinkiewicz--Zygmund theorem cannot be applied directly to obtain the desired vector-valued extension. 

The above strategy does work, however, when $\mu_j < R^{-1}$. Since the $x$-integration is localised to $Q_R$, in this case uncertainty principle considerations ensure that $E_{S_j'(\bdeta_j)}f_{j,\bdeta_j}$ is essentially constant in $\bdeta_j \in [-\mu_j, \mu_j]^{m_j}$. This allows for an application of the Marcinkiewicz--Zygmund theorem as in \cite[Appendix]{Tao2020}: see Lemma~\ref{lem: uncer red} for details. On the other hand, when $\mu_j \geq R^{-1}$ it is possible to apply Kakeya--Brascamp--Lieb inequalities to reduce the scale $R$, following an induction-on-scale procedure similar to that used in \cite{BCT2006, BBFL2018} (but adapted to take into account the localisation of the supports of the $f_j$). Combining these two observations, one may formulate a refined induction-on-scale procedure, which takes into account the relative sizes of the $\mu_j$ and $R^{-1}$: see Lemma~\ref{lem: recursive step}. As in \cite{BBFL2018, BCT2006}, the induction-on-scale procedure yields the desired estimate.




\subsection{Notational conventions} We shall say a constant is \textit{admissible} if it depends either only on a choice of dimension $n$ or a choice of ensemble $\sE$. In particular, any admissible constant is independent of parameters such as $R \geq 1$ or the $\mu_{j,\ell} > 0$. We will frequently let $C$ denote a choice of admissible constant, whose value may change from line to line. Given a list of objects $L$ and real numbers $A$, $B \geq 0$, we write $A \lesssim_L B$ or $B \gtrsim_L A$ to indicate $A \leq C_L B$ for some constant $C_L$ which depends only items in the list $L$ and either a choice of dimension $n$ or a choice of ensemble $\sE$. We write $A \sim_L B$ to indicate $A \lesssim_L B$ and $B \lesssim_L A$. Given $d \in \N$ we let $I_d$ denote the $d \times d$ identity matrix.




\section{Kakeya--Brascamp--Lieb theory}\label{sec:KBL}




\subsection{Regularised Brascamp--Lieb revisited}
Let $2 \leq k \leq n$, $1 \leq n_j \leq n$ for $1 \leq j \leq k$ and $(\bL, \bp)$ a Brascamp--Lieb datum as in \S\ref{subsec: reg BL}. Given $R \geq 1$ and $0 \leq \lambda<R$ dyadic, we let $\BLscale{\bL}{\bp}{\lambda}{R}$ denote the best constant $C>0$ for which the associated generalised Brascamp--Lieb inequality
    \begin{align}\label{eq: BLdef}
        \int_{Q_R}\prod_{j=1}^k (f_j\circ L_j)^{p_j} \leq C\prod_{j=1}^k\Big(\int_{\R^{n_j}}f_j\Big)^{p_j}
    \end{align}
holds for all non-negative functions $f_j \in L^1(\R^{n_j})$ that are constant on cubes of sidelength $\lambda$ in the fixed dyadic grid $[0,\lambda)^{n_j} + \lambda \Z^{n_j}$. Inequalities of this form were first introduced in \cite{BCCT2008} and can be seen as a generalisation of the classical scale-invariant Brascamp--Lieb inequalities, which correspond to taking $\lambda \to 0$, $R \to \infty$. Regarding the latter, we let $\BL{\bL}{\bp}$ denote the classical Brascamp--Lieb constant, which is the best constant $C>0$ for which \eqref{eq: BLdef} holds for all $R \geq 1$ and  all non-negative functions $f_j \in L^1(\R^{n_j})$ (without any local constancy hypothesis). 

The case $\lambda=1$ of the above is often referred to as \textit{regularised}.\footnote{Note that if $0 < \lambda \leq R < \infty$, one can always reduce to the case $\lambda=1$ by the scaling relation
\begin{equation*}
    \BLscale{\bL}{\bp}{\lambda}{R}=\lambda^{n-\sum_{j=1}^k p_jn_j}\BLscale{\bL}{\bp}{1}{R\lambda^{-1}}.
\end{equation*}}
Maldague \cite[Theorem 1.1]{Maldague2022} recently showed that
\begin{equation}\label{eq:charact BL 1R}
    \BLscale{\bL}{\bp}{1}{R} \sim_{\bL,\bp} R^{\alpha(\bL,\bp)}
\quad
\textrm{for}
\quad
\alpha(\bL,\bp):=\sup_{V\leq \R^n}\Big(\dim(V)-\sum_{j=1}^k p_j\dim(L_jV)\Big).
\end{equation}
Note that the constant $\BLreg{\bL}{\bp}$ introduced in \S\ref{subsec: reg BL} satisfies 
\begin{equation*}
   \BLreg{\bL}{\bp}=\sup_{R>1} \BLscale{\bL}{\bp}{1}{R}; 
\end{equation*}
 in particular, one can recover \eqref{eq:BLreg finite} from the characterisation \eqref{eq:charact BL 1R} for $\BLscale{\bL}{\bp}{1}{R}$. We remark that the classical Brascamp--Lieb constant $\BL{\bL}{\bp}$ is finite if and only if $\alpha(\bL, \bp)=0$ and $\sum_{j = 1}^k p_j n_j = n$: see \cite{BCCT2008}. In this case, we in fact have
 \begin{equation*}
     \BLreg{\bL}{\bp} = \BL{\bL}{\bp}
 \end{equation*}
 as a consequence of a simple limiting argument.



\subsection{Regularised Brascamp--Lieb  and transversality conditions} Here we compare the hypotheses of Theorem~\ref{thm: nested} and Theorem~\ref{thm: Bejenaru} and, in particular, address the comments made in Remark~\ref{rmks: main} b).

\begin{proposition}\label{prop: reg to normal}
    Let $\bp_k = (p_k)_{j=1}^k$ where $p_k := q_k/2=\frac{1}{k-1}$ and suppose that $L_j \colon \R^n\to\R^{n_j}$, $1\leq j\leq k$, are linear surjections. Then the following are equivalent:
    \begin{enumerate}[a)]
    \item $\BLreg{\bL}{\bp_k} <\infty$;
    \item $\displaystyle \Big|\bigwedge_{j=1}^k \ker(L_j)\Big|>0$.
\end{enumerate}
\end{proposition}
\begin{proof}
For $V \leq \R^n$, we first note that
\begin{equation}\label{eq: rank-nullity}
    \dim(L_j(V))=\dim(V)-\dim(V \cap \ker(L_j)),
\end{equation}
 by the rank--nullity theorem applied to the restriction $L_j |_V$, $1 \leq j \leq k$. We define the subspace $W:=\ker(L_1)+ \cdots +\ker(L_k)$. \medskip

\noindent \textit{a) $\Rightarrow$ b)}. Assume b) fails. Applying \eqref{eq: rank-nullity} to $W$ as defined above, we deduce that
\begin{equation}\label{eq: rank-nullity 2}
    \dim(L_j(W))=\dim(W)-\dim(\ker(L_j)),
\end{equation}
since $\ker (L_j) \leq W$ for all $1 \leq j \leq k$. The failure of b) implies that $\dim(W)<\sum_{j=1}^k\dim(\ker(L_j))$ and applying this to \eqref{eq: rank-nullity 2} yields
\begin{equation*}
    \sum_{j=1}^k\dim(L_j (W))=k\dim(W)-\sum_{j=1}^k\dim(\ker(L_j))<(k-1)\dim(W).
\end{equation*}
This means that $W$ violates the finiteness characterisation  \eqref{eq:BLreg finite} for the datum $(\bL,\bp_k)$, implying that a) fails.\medskip

\noindent \textit{b) $\Rightarrow$ a)}. Assume that b) holds, so that $W=\bigoplus_{j=1}^k\ker(L_j)$. Given $V\leq\R^n$, we have $\bigoplus_{j=1}^k (V\cap\ker(L_j)) \leq V\cap W$ and therefore
\begin{equation*}
\sum_{j=1}^k\dim(V\cap\ker(L_j))=\dim\bigg(\bigoplus_{j=1}^k V\cap\ker(L_j)\bigg)\leq\dim(V\cap W).
\end{equation*} 
Combining the above inequality with \eqref{eq: rank-nullity}, we obtain
\begin{align*}
    \sum_{j=1}^k\dim(L_jV)&=\sum_{j=1}^k(\dim(V)-\dim(V\cap\ker(L_j))\\
    &\geq k\dim(V)-\dim(V\cap W)\\
    &\geq (k-1)\dim(V).
\end{align*}
Since $V \leq \R^n$ was arbitrary, the characterisation \eqref{eq:BLreg finite} implies that a) holds.
\end{proof}

\begin{remark}\label{rem: qual to quant}
    Suppose the maps $L_j$ are orthogonal projections. By invoking \cite[Proposition 1.2]{BB2010}, one may deduce that
    \begin{equation*}
        \BLreg{\bL}{\bp_k}\sim \Big|\bigwedge_{j=1}^k\ker(L_j)\Big|^{-\frac{1}{k-1}},
    \end{equation*}
    which is a quantitative variant of Proposition \ref{prop: reg to normal}. Since in this paper we only need finiteness of such transversality quantites, and not uniform boundedness, we leave the proof of this more quantitative result to the interested reader.
    
    
    
\end{remark}
\begin{corollary}\label{cor: BL to Bejenaru}
Let $2 \leq k \leq n$, $\sE = ((S_j, S_j')_{j =1}^k, \bq_k)$ be a transverse ensemble in $\R^n$. Then 
\begin{equation}\label{eq: BL to Bejenaru}
        \BLreg{\bL(\sE)}{\bp_k}<\infty \quad \Longleftrightarrow \quad  \Big|\bigwedge_{j = 1}^k N_{\xi_j'}S_j'\Big| >0.
    \end{equation}
    Here the $\xi_j'$ are as defined in \eqref{eq: trans}. In particular, Theorem~\ref{thm: nested} implies Theorem~\ref{thm: Bejenaru}.
\end{corollary}

\begin{proof}
For each $1\leq j\leq k$, let $\Sigma_j \colon U_j\rightarrow \R^n$ be the parametrisation of $S_j$, and let $M_j\subset U_j$ be a submanifold of $\R^{\dim(S_j)}$, parametrised by $\gamma_j:U_j'\rightarrow U_j$, such that $\Sigma_j(M_j)=S_j'$. Define $\Sigma_j':=\Sigma_j\circ\gamma_j:U_j'\rightarrow\R^n$. Then, by the definition of $\bL(\sE)$ we have that $L_j(\sE)=\partial_s\Sigma_j'(0)^\top$. Since $\ker(\partial_s\Sigma_j'(0)^\top)=N_{\xi_j'}S_j'$, we then see that the equivalence \eqref{eq: BL to Bejenaru} is merely a rephrasing of Proposition~\ref{prop: reg to normal} for this particular choice of Brascamp--Lieb datum.
\end{proof}


\subsection{Brascamp--Lieb constants and subensembles}\label{sec: BL ensemble}

Fix $2 \leq k \leq n$ and let $\mathscr{E}=((\cS_j)_{j=1}^k, \bq)$ be a transverse ensemble in the sense of Definition \ref{def: nested}, with $\cS_j=(S_j,(S_{j,\ell})_{\ell=1}^{r_j})$ and $\bq=(q_j)_{j=1}^k \subset (0,2]$. We say that $\mathscr{F}=( (\widetilde{\cS}_j)_{j=1}^k, \bq)$ is a \textit{subensemble} of $\mathscr{E}$ if $\widetilde{\cS}_j=(S_j, (S_{j,\ell})_{\ell=1}^{\tilde{r}_j})$ with $\tilde{r}_j \leq r_j$. Recall we allow values $\tilde{r}_j=0$, in which case $\widetilde{\cS}_j$ consists only of $S_j$.

\begin{proposition}\label{prop: BL ensemble}
    Let $\mathscr{E}$ be a transverse ensemble in $\R^n$. Then any subensemble $\mathscr{F}$ of $\mathscr{E}$ is also a transverse ensemble.
\end{proposition}
\begin{proof}
    It suffices to verify that $\BLreg{\bL(\mathscr{F})}{\bp} < \infty$ with $\bL(\mathscr{F})$ as in \eqref{eq: map ensemble def}. In view of \eqref{eq: chain of param}, for each $1 \leq j \leq k$ we have that
    \begin{equation*}
        \Sigma_{j,r_j}= \Sigma_{j,\tilde{r}_j} \circ \gamma_{j,\tilde{r}_j,r_j} \quad \text{ for }\,\, \gamma_{j,\tilde{r}_j,r_j}:= \gamma_{j,\tilde{r}_{j} +1} \circ \cdots \circ \gamma_{j, r_j};
    \end{equation*}
    here, for the case $\tilde{r}_j=r_j$, we simply have $\gamma_{j,r_j,r_j} :=\mathrm{Id}$. Thus, 
    \begin{equation*}
        L_j(\mathscr{E})=\partial_{s} \Sigma_{j,r_j}(0)^\top =  (\partial_s \gamma_{j,\tilde{r}_j, r_j}(0) )^{\top} (\partial_{\tilde{s}} \Sigma_{j,\tilde{r}_j} (0))^\top = (\partial_s \gamma_{j,\tilde{r}_j, r_j}(0) )^{\top} L_j(\mathscr{F})
    \end{equation*}
    where we have used that $\gamma_{j,\ell}(0)=0$ for all $1 \leq \ell \leq r_j$. It is then clear that $\dim(L_j(\mathscr{E}) V) \leq \dim(L_j(\mathscr{F})V)$ for any $V \leq \R^n$ and, in view of the characterisation \eqref{eq:BLreg finite} for $\BLreg{\bL}{\bp}$, we conclude that the finiteness of $\BLreg{\bL(\mathscr{E})}{\bp}$ implies that of $\BLreg{\bL(\mathscr{F})}{\bp}$.
\end{proof}
\begin{remark}
    Similarly to Remark \ref{rem: qual to quant}, if the linear surjections $L_j(\sE)$ and $L_j(\sF)$ are orthogonal projections for $1\leq j\leq k$, then one can show that $\BLreg{\bL(\sF)}{\bp}\lesssim \BLreg{\bL(\sE)}{\bp}$, where the implicit constant does not depend on the choice of ensemble. This refined quantitative result is not required for our purposes, and so we leave the proof to the interested reader. 
\end{remark}



\subsection{Kakeya--Brascamp--Lieb inequalities} Here we are interested in a class of \textit{Kakeya--Brascamp--Lieb inequalities}, which can be thought of as perturbed variants of the regularised Brascamp--Lieb inequalities introduced above. The following result is due to Maldague \cite{Maldague2022}. 

\begin{theorem}[Kakeya--Brascamp--Lieb {\cite[Theorem 1.2]{Maldague2022}}]\label{thm: Mald}
Let $(\bL,\bp)$ be a Brascamp--Lieb datum
such that
$\BLreg{\bL}{\bp} < \infty$ . Then there exists $\nu > 0$ such that the following holds. For all $\varepsilon>0$, there exists a constant $C_\varepsilon \geq 1$ such that the inequality
\begin{equation*}    \int_{Q_R}\prod_{j=1}^k\Big|\sum_{T_j\in\T_j}c_{T_j}\chi_{T_j}\Big|^{p_j}\leq C_\varepsilon R^{\varepsilon} \lambda^{n-\varepsilon} \prod_{j=1}^k\Big(\sum_{T_j\in\T_j}|c_{T_j}|\Big)^{p_j}
\end{equation*}
holds for all $R\geq 1$ and $0<\lambda<R$, whenever the $\T_j$ are countable collections of infinite slabs of width $\lambda$ whose core $n_j$-dimensional planes are, modulo translations, within a distance $\nu>0$ (in the grassmannian sense) from the fixed subspace $X_j:=\ker L_j$, and $(c_{T_j})_{T_j \in \T_j} \in \ell^1(\T_j)$ for $1 \leq j \leq k$.
\end{theorem}
Here we say $T \subseteq \R^n$ \textit{is an infinite slab of width} $r > 0$ if\footnote{Here and below we let $|x|_{\infty} := \max_{1 \leq \ell \leq m} |x_{\ell}|$ for $x = (x_1, \dots, x_m) \in \R^m$.} 
\begin{equation*}
   T = \{x \in \R^n : |Lx - v|_{\infty} < r\} 
\end{equation*}
for some linear surjective map $L \colon \R^n \to \R^m$ and $v \in \R^m$. In this case, the affine subspace $V := \{x \in \R^n : Lx = v\}$ is called the \textit{core plane} of $T$. We note that endpoint variants of Theorem~\ref{thm: Mald} can be found in \cite{Zhang2018, Zorin-Kranich2020}. The case in which $\dim \ker L_j =1$ for all $1 \leq j \leq k$ corresponds to the $k$-linear Kakeya inequality from \cite[Theorem 5.1]{BCT2006}. 

We remark that there are some differences between the statement of Theorem~\ref{thm: Mald} and that of \cite[Theorem 1.2]{Maldague2022}. 
\begin{itemize}
    \item Theorem~\ref{thm: Mald} is a rescaled version of \cite[Theorem 1.2]{Maldague2022}. The parameter $\delta$ in \cite[Theorem 1.2]{Maldague2022} corresponds to $\lambda/R$ in Theorem~\ref{thm: Mald}.
    \item The coefficients $(c_{T_j})_{T_j \in \T_j}$ are taken to all be $1$ in \cite[Theorem 1.2]{Maldague2022}. However, the result for general coefficients follows from a standard approximation argument: see \cite{BCT2006}. 
    \item Theorem~\ref{thm: Mald} is stated in terms of $\BLreg{\bL}{\bp}$ whilst  \cite[Theorem 1.2]{Maldague2022} is stated in terms of the exponent 
    \begin{equation*}
       \tilde{\alpha}(\bL, \bp) := \sup_{V\leq \R^n}\Big(\dim(V)-\sum_{j=1}^k p_j\dim(\pi_jV)\Big)
    \end{equation*}
    where $\pi_j \colon \R^n \to \R^n$ is the orthogonal projection onto the subspaces $(X_j)^{\perp}$ for $1 \leq j \leq n$. To relate the former from the latter, one observes that the hypothesis $\BLreg{\bL}{\bp}<\infty$ implies $\tilde{\alpha}(\bL, \bp) = 0$. Indeed, $\ker \pi_j = X_j = \ker L_j$ and, given any $V \leq \R^n$, it follows from rank-nullity as in \eqref{eq: rank-nullity} that $\dim \pi_j(V) = \dim L_j(V)$ for all $1 \leq j \leq n$. Consequently, $\tilde{\alpha}(\bL, \bp) = \alpha(\bL, \bp)$ and the desired implication follows from the characterisation \eqref{eq:charact BL 1R}.
\end{itemize}


\section{Fourier extension preliminaries}\label{sec: extension}




\subsection{Parametrising the neighbourhoods}\label{subsec: parametrising} Let $\cS = (S_0, (S_{\ell})_{\ell = 1}^r)$ be a nested family of submanifolds of $\R^n$, with $S_0 = \Sigma_0(U_0)$ a $d_0$-dimensional submanifold of $\R^n$. As in \S\ref{subsec: main result}, we have a family of open neighbourhoods $U_{\ell} \subseteq \R^{d_{\ell}}$ of the origin, with $d_\ell=\dim(S_\ell)$ satisfying $d_r < \cdots < d_1 < d_0$, and parametrisations $\gamma_{\ell} \colon U_{\ell} \to U_{\ell-1}$  for $1 \leq \ell \leq r$ satisfying $\gamma_{\ell}(0) = 0$ such that
\begin{equation*}
    S_{\ell} = \Sigma_{\ell}(U_{\ell}) \quad \textrm{where} \quad \Sigma_{\ell} := \Sigma_0 \circ \sigma_{\ell} \quad \textrm{for} \quad \sigma_{\ell} := \gamma_1 \circ \cdots \circ \gamma_{\ell}.
\end{equation*}
Each set $M_\ell:=\sigma_\ell(U_\ell)$ is a $d_\ell$-dimensional submanifold of $\R^{d_0}$ and $S_\ell=\Sigma_0(M_\ell)$. We define 
\begin{equation*}
    m_{\ell} := \codim{S_{\ell}}{S_{\ell-1}} = d_{\ell - 1} - d_{\ell} \qquad \textrm{and} \qquad c_{\ell} := \codim{S_{\ell}}{S_0} = d_0 - d_{\ell}
\end{equation*}
for $1 \leq \ell \leq r$. Often we drop the subscript $0$ and write $S := S_0$, $d := d_0$, $\Sigma := \Sigma_0$ and $U := U_0$.

We begin with some (unfortunately rather technical) definitions which allow us to work inductively with the nested framework of Theorem~\ref{thm: nested}. With the above setup, we may define
\begin{equation*}
    M_{k,\ell} := \gamma_{k,\ell}(U_\ell) \subseteq U_k \quad \textrm{where} \quad \gamma_{k,\ell} := \gamma_{k+1} \circ \cdots \circ \gamma_\ell \qquad \textrm{for $k + 1 \leq \ell \leq r$,}
\end{equation*}
so that each $M_{k,\ell}$ is a smooth $d_{\ell}$-dimensional submanifold of $\R^{d_k}$. Recall that the support conditions in Theorem~\ref{thm: nested} are defined in terms of neighbourhoods of the submanifolds $M_{\ell}$, and note that $M_{\ell} = M_{0,\ell}$ with the above notation. In order to argue inductively, we shall work more generally with neighbourhoods of the $M_{k, \ell}$ for $0 \leq k < \ell$. 

To describe the neighbourhoods of the $M_{k, \ell}$, we introduce a family of parametrising maps $\Phi_{k,\ell}$. Given $1\leq \ell
\leq r$, we let $G_\ell \colon U_\ell\rightarrow \mathrm{Mat}(\R, d_{\ell-1}\times m_\ell)$ denote a $(d_{\ell-1}\times m_\ell)$-matrix-valued function such that the columns of $G_\ell(s_\ell)$ form an orthogonal basis for $N_{\gamma_\ell(s_\ell)}M_{\ell-1,\ell}\subset U_{\ell-1}$. Let $\bar{\mu}_{\circ} > 0$ be an admissible parameter, chosen sufficiently small for the purposes of the forthcoming argument. Given $1 \leq k < \ell \leq r$, throughout the following we write
\begin{equation*}
 V_{k,\ell} := \prod_{t = k+1}^{\ell} (-\bar{\mu}_{\circ},\bar{\mu}_{\circ})^{m_t} \quad \textrm{and} \quad   \bdeta_{k,\ell} = (\eta_{k+1},\dots, \eta_\ell) \quad \textrm{for $\bdeta_{k,\ell} \in V_{k,\ell}$.}
\end{equation*}
For fixed $1 \leq \ell \leq r$, we recursively define maps $\Phi_{k,\ell} \colon U_\ell\times V_{k,\ell} \rightarrow U_k$ for $0 \leq k \leq \ell$, where we interpret $U_{\ell} \times V_{\ell, \ell}$ as $U_{\ell}$, by setting $\Phi_{\ell,\ell} \colon U_\ell  \to U_{\ell}$ to be the identity map and then defining
\begin{equation}\label{eq: phi def}
    \Phi_{k-1,\ell}(s_\ell;\bdeta_{k-1,\ell}) :=\gamma_k\circ\Phi_{k,\ell}(s_\ell;\bdeta_{k,\ell})+G_{k}\circ\Phi_{k,\ell}(s_\ell;\bdeta_{k,\ell})\eta_k^\top 
\end{equation}
for $1 \leq k \leq \ell$, where here $\bdeta_{k-1,\ell} = (\eta_k, \bdeta_{k,\ell})$ and $\bdeta_{\ell,\ell}$ is a null variable.  Since each set $\gamma_k(U_k)$ is, by hypothesis, compactly contained in $U_{k-1}$, provided $\bar{\mu}_\circ > 0$ is chosen sufficiently small, the maps $\Phi_{k,\ell}$ are well-defined and indeed map into $U_k$. By a simple induction argument,
\begin{equation}\label{eq: Phi vs gamma}
    \Phi_{k,\ell}(s_{\ell}; \bzero) = \gamma_{k,\ell}(s_{\ell}),
\end{equation}
where $\gamma_{\ell,\ell} : =\mathrm{Id}_{U_\ell}$.

We may equivalently write \eqref{eq: phi def} as
\begin{equation}\label{eq: phi presplit}
\Phi_{k-1,\ell}(s_\ell;\bdeta_{k-1,\ell})=\Phi_{k-1,k}(\Phi_{k,\ell}(s_\ell;\bdeta_{k,\ell});\eta_k).
\end{equation}
Iterating the above observation, we shall derive the following formula.
\begin{lemma}
    For all $0\leq k\leq \ell\leq r$, $s_r\in W_r$, and $\bdeta_{k,r}\in V_{k,r}$ we have
    \begin{equation}\label{eq: phi split}
    \Phi_{k,r}(s_r;\bdeta_{k,r})=\Phi_{k,\ell}(\Phi_{\ell,r}(s_r;\bdeta_{\ell,r});\bdeta_{k,\ell}),
\end{equation}
where we have written $\bdeta_{k,r}=(\bdeta_{k,\ell}, \bdeta_{\ell,r})$.
\end{lemma}
\begin{proof}
    The case when $k=\ell=r$ is trivial. Now assume, by way of inductive hypothesis, that \eqref{eq: phi split} holds for some $1\leq k\leq r$ and all $\ell\in\N$ such that $k\leq\ell\leq r$. 
    
    Using the induction hypothesis, we shall show 
    \begin{equation}\label{eq: phi split 1}
        \Phi_{k-1,r}(s_r;\bdeta_{k-1,r})=\Phi_{k-1,\ell}(\Phi_{\ell,r}(s_r;\bdeta_{\ell,r});\bdeta_{k-1,\ell})
    \end{equation}
    holds for all $k-1 \leq \ell \leq r$. If $\ell = k-1$, then this result is a trivial consequence of the definitions, so we assume $k \leq \ell \leq r$. By \eqref{eq: phi presplit}, we deduce that 
    \begin{equation}\label{eq: phi split 2}
         \Phi_{k-1,r}(s_r;\bdeta_{k-1,r})=\Phi_{k-1,k}(\Phi_{k,r}(s_r;\bdeta_{k,r});\eta_k).
    \end{equation}
    Our inductive hypothesis then implies that
    \begin{equation}\label{eq: phi split 3}
         \Phi_{k-1,r}(s_r;\bdeta_{k-1,r})=\Phi_{k-1,k}(\Phi_{k,\ell}(\Phi_{\ell,r}(s_r;\bdeta_{\ell,r});\bdeta_{k,\ell});\eta_k).
    \end{equation}
    Applying \eqref{eq: phi presplit} again, here with $s_\ell=\Phi_{\ell,r}(s_r;\bdeta_{\ell,r})$, then yields
    \begin{equation}\label{eq: phi split 4}
         \Phi_{k-1,k}(\Phi_{k,\ell}(\Phi_{\ell,r}(s_r;\bdeta_{\ell,r});\bdeta_{k,\ell});\eta_k)=\Phi_{k-1,\ell}(\Phi_{\ell,r}(s_r;\bdeta_{\ell,r});\bdeta_{k-1,\ell}).
    \end{equation}
    Combining \eqref{eq: phi split 2}, \eqref{eq: phi split 3} and \eqref{eq: phi split 4}, we obtain \eqref{eq: phi split 1}, which closes the induction.
\end{proof}

 For each $1\leq k\leq r$, the derivative  $\partial_{s,\eta_k}\Phi_{k-1,k}(s; \eta_k)|_{\eta_k = 0} = [\partial_s \gamma_k(s) \,\,\,  G_k(s)]$ has full rank over all of $U_k$. By the inverse function theorem, provided  $\bar{\mu}_{\circ} > 0$ is sufficiently small, there exist open neighbourhoods of the origin $W_{k}\subseteq U_{k}$, $1\leq k\leq r$, such that the restricted maps $\Phi_{k-1,k} \colon W_k\times (-\bar{\mu}_{\circ},\bar{\mu}_{\circ})^{m_k}\rightarrow W_{k-1}$ are diffeomorphisms.\footnote{Note that the sets $W_j$ appear in \textit{both} the domain and codomain of these restricted mappings. This `compatibility' can be achieved by repeatedly pruning the sets arising from direct application of the inverse function theorem; we leave the details to the dedicated reader.} By iterating \eqref{eq: phi presplit}, the same is true of the restrictions 
 \begin{equation}\label{eq: diff chain}
      \Phi_{k,\ell} \colon W_{\ell}\times V_{k,\ell} \rightarrow W_k, \qquad 0\leq k\leq\ell\leq r. 
 \end{equation}
We later refer to these restricted mappings as the \textit{diffeomorphism chain} for $\cS$.

Let $C_{\circ}\geq 1$ be an admissible constant, chosen sufficiently large for the forthcoming purposes of the argument and define $\mu_{\circ} := C_{\circ}^{-9/8} \bar{\mu}_{\circ}$. We later refer to this quantity as the \textit{threshold width} for $\cS$. \label{page: threshold} Given $\bmu=(\mu_1, \dots,\mu_\ell)$ with $\max_{1 \leq \ell \leq r} \mu_{\ell} \leq C_{\circ}\mu_{\circ}$, define the set
 \begin{equation}\label{eq: Omega def}
     \Omega_{k,\ell}(\bmu):=\Phi_{k,\ell}(W_{\ell}\times P_{k,\ell}(\bmu)) \subseteq W_k \quad \textrm{where} \quad P_{k,\ell}(\bmu):=\prod_{t=k+1}^\ell[-C_{\circ}^{1/8}\mu_t,C_{\circ}^{1/8}\mu_t]^{m_t}. 
 \end{equation}
The following lemma tells us that the $\Phi_{k,\ell}$ can indeed be thought of as parametrising the (anisotropic) neighbourhoods of the submanifolds $M_{k,\ell}$.

 \begin{lemma}\label{lem: Omega incl}
     Provided $C_{\circ}\geq 1$ is chosen sufficiently large,
     \begin{equation}
\bigcap_{t=k+1}^\ell\cN_{\mu_t}M_{k,t}\cap W_k\subseteq\Omega_{k,\ell}(\bmu)\subseteq \bigcap_{t=k+1}^\ell\cN_{C_{\circ}\mu_t}M_{k,t}\cap W_k\label{eq: omega incl}
     \end{equation}
     holds for $1\leq k\leq \ell\leq r$ and $\bmu = (\mu_1, \dots ,\mu_\ell)$ such that $0 < \mu_1 \leq \cdots \leq \mu_r\leq  C_{\circ}\mu_\circ$. 
 \end{lemma}
 
When $k = \ell$, the left- and right-hand sides of \eqref{eq: omega incl} are understood to equal $W_\ell$. For $\rho > 0$, it is convenient to write
 \begin{equation}\label{eq: neighbourhood notation}
     \fN_{k,\ell}(\bmu; \rho) := \bigcap_{t=k+1}^\ell\cN_{\rho \mu_t}M_{k,t}\cap W_k,
 \end{equation}
so that the left and right-hand sets in \eqref{eq: omega incl} correspond to $\fN_{k,\ell}(\bmu; 1)$ and $\fN_{k,\ell}(\bmu; C_{\circ})$.
 
 \begin{proof}[Proof (of Lemma~\ref{lem: Omega incl})]  For the purpose of this proof, we shall rescale $\bmu$ and temporarily redefine $P_{k,\ell}(\bmu):=\prod_{t=k+1}^\ell[-\mu_t,\mu_t]^{m_t}$. Let $C \geq 1$ be a fixed constant, chosen to satisfy the forthcoming requirements of the proof. Fixing $1 \leq \ell \leq r$, we shall show,  under the modified definition,
 \begin{equation}\label{eq: inductive hyp}
         \fN_{k,\ell}(\bmu; \rho_k^{-1}) \subseteq \Omega_{k,\ell}(\bmu) \subseteq \fN_{k,\ell}(\bmu; \rho_k) \qquad \textrm{for} \qquad \rho_k := C^{\ell - k}
     \end{equation}
 holds for all $1 \leq k \leq \ell$, using induction on $k$. Once this is established, we may simply take $C_{\circ} \geq  C^{8n}$ and replace $\bmu$ with $C_{\circ}^{1/8}\bmu$ to conclude the desired result. The claim holds vacuously for $k=\ell$, which acts as the base case. Now let $2 \leq k \leq \ell$ and suppose \eqref{eq: inductive hyp} holds for this value of $k$. It suffices to show
     \begin{equation}\label{eq: inductive incl}
          \fN_{k-1,\ell}(\bmu; C^{-1}\rho_k^{-1}) \subseteq\Omega_{k-1,\ell}(\bmu)\subseteq \fN_{k-1,\ell}(\bmu; C \rho_k).
     \end{equation}
     
     For $\rho > 0$, define the auxiliary sets
     \begin{equation*}
       \widetilde{\fN}_{k-1,\ell}(\bmu; \rho) :=  M_{k-1,k}\cap\bigcap_{t=k+1}^\ell\cN_{\rho \mu_t}M_{k-1,t}\cap W_{k-1}.
     \end{equation*}
     The proof of \eqref{eq: inductive incl} can be reduced to showing
    \begin{equation}\label{eq: gamma incl}
        \widetilde{\fN}_{k-1,\ell}(\bmu; 2C^{-1} \rho_k^{-1}) \subseteq \gamma_k(\Omega_{k,\ell}(\bmu)) \subseteq \widetilde{\fN}_{k-1,\ell}(\bmu; 2^{-1} C \rho_k).
    \end{equation}
    We establish this reduction in two stages.\medskip

    Assuming the first inclusion in \eqref{eq: gamma incl}, we show the first inclusion in \eqref{eq: inductive incl}. Let $x \in \fN_{k-1,\ell}(\bmu; C^{-1}\rho_k^{-1})$ and note, since $x\in W_{k-1}$, that there exists $s\in W_\ell$  and $ \bdeta_{k-1,\ell} = (\eta_k, \bdeta_{k,\ell}) \in V_{k-1, \ell}$ such that 
     \begin{equation}\label{eq: x to xzero}
         x=\Phi_{k-1,\ell}(s;\bdeta_{k-1,\ell}) =\gamma_k\circ\Phi_{k,\ell}(s;\bdeta_{k,\ell})+G_{k}\circ\Phi_{k,\ell}(s;\bdeta_{k,\ell})\eta_k^\top.
     \end{equation}
     where we have used the recursive definition \eqref{eq: phi def}. Observe that:
     \begin{itemize}
         \item $x_0:=\gamma_k \circ \Phi_{k,\ell}(s;\bdeta_{k,\ell})\in M_{k-1,k}$ and $x_0 = \Phi_{k-1, \ell}(s; (0,\bdeta_{k,\ell})) \in W_{k-1}$;
         \item $x\in \cN_{C^{-1} \rho_k^{-1} \mu_t} M_{k-1,t}$ for $k \leq t \leq \ell$;
         \item $x-x_0 = G_{k}\circ\Phi_{k,\ell}(s;\bdeta_{k,\ell})\eta_k^\top \in N_{x_0}M_{k-1,k}$. 
     \end{itemize}
     We therefore conclude that $|x-x_0| = |\eta_k|\leq  C^{-1}\rho_k^{-1}\mu_k$. By the ordering $\mu_1 \leq \dots \leq \mu_r$ and the triangle inequality, $x_0 \in \cN_{2 C^{-1}\rho_k^{-1}\mu_t}M_{k-1,t}$ for each $k\leq t\leq \ell$, and hence $x_0\in \widetilde{\fN}_{k-1,\ell}(\bmu; 2 C^{-1}\rho_k^{-1})$. The first inclusion in \eqref{eq: gamma incl} therefore implies that $x_0 \in \gamma_k(\Omega_{k,\ell}(\bmu))$. By the definition of $x_0$ and the injectivity of $\gamma_k$, we therefore deduce that $\Phi_{k,\ell}(s;\bdeta_{k,\ell}) \in \Omega_{k,\ell}(\bmu)$. By the injectivity of $\Phi_{k,\ell}$ and the definition of $\Omega_{k,\ell}(\bmu)$, we further deduce that $\bdeta_{k,\ell} \in P_{k,\ell}(\bmu)$. Since we have already shown $|\eta_k| \leq \mu_k$, we conclude that $\bdeta_{k-1,\ell} \in P_{k-1,\ell}(\bmu)$. Thus, \eqref{eq: x to xzero} implies $x \in \Omega_{k-1, \ell}(\bmu)$, as required.\medskip

    Assuming the second inclusion in \eqref{eq: gamma incl}, we show the second inclusion in \eqref{eq: inductive incl}. Let $x\in\Omega_{k-1,\ell}(\bmu)$, so that $x=\Phi_{k-1,\ell}(s;\bdeta_{k-1,\ell})$ for some $s\in W_\ell$ and $\bdeta_{k-1,\ell}\in P_{k-1,\ell}(\bmu)$. For $x_0:=\gamma_k \circ \Phi_{k,\ell}(s;\bdeta_{k,\ell})$ as before, we automatically have $|x-x_0| = |\eta_k| \leq \mu_k$. Furthermore, $x_0 \in \gamma_k(\Omega_{k,\ell})$ and so \eqref{eq: gamma incl} implies that $x_0\in \widetilde{\fN}_{k-1,\ell}(\bmu; 2^{-1}C\rho_k)$. By the ordering of $\bmu$ and the triangle inequality, provided $C$ is chosen sufficiently large we have $x\in \fN_{k-1,\ell}(\bmu; C \rho_k)$, as required.\medskip
     
     It remains to verify \eqref{eq: gamma incl}. Let $x \in \widetilde{\fN}_{k-1,\ell}(\bmu; 2 C^{-1}\rho_k^{-1})$. Since $x \in M_{k-1, k} \cap W_{k-1}$, there exists some $s_k \in W_k$ such that $x = \gamma_k(s_k)$. On the other hand, fixing $k+1\leq t\leq\ell$, there exists some $\tilde{x} \in M_{k-1,t}$ such that $|x - \tilde{x}| < 2 C^{-1}\rho_k^{-1} \mu_t$. Moreover, $\tilde{x} = \gamma_k(\tilde{s}_k)$ for some $\tilde{s}_k \in M_{k,t}$. We therefore have $|\gamma_k(s_k) - \gamma_k(\tilde{s}_k)| < 2 C^{-1}\rho_k^{-1} \mu_t$. Since the parametrisations $\gamma_k$ are regular, provided that $C \geq 1$ is chosen sufficiently large, it follows that $|s_{k}-\tilde{s}_{k}|< \rho_k^{-1}\mu_t$. This implies that $s_{k}\in\cN_{\rho_k^{-1}\mu_t}M_{k,t}$. Since this holds for all $k+1\leq t\leq\ell$, we have shown 
     \begin{equation*}
         \widetilde{\fN}_{k-1,\ell}(\bmu; 2 C^{-1}\rho_k^{-1}) \subseteq \gamma_{k}\big(\fN_{k,\ell}(\bmu; \rho_k^{-1})\big) \subseteq \gamma_k\big(\Omega_{k,\ell}(\bmu)\big),
     \end{equation*}
    where the last inequality is due to the induction hypothesis \eqref{eq: inductive hyp}.

    Now suppose $x \in \gamma_k\big(\Omega_{k,\ell}(\bmu)\big)$. By the induction hypothesis \eqref{eq: inductive hyp}, there exists some $s_k \in \fN_{k,\ell}(\bmu; C\rho_{k+1})=\fN_{k,\ell}(\bmu; \rho_{k})$ such that $x = \gamma_k(s_k)$. Fixing $k+1 \leq t \leq \ell$, it follows that $s_k\in \cN_{\rho_k\mu_t}M_{k,t}\cap W_{k}$ and so there exists $\tilde{s}_{k}\in M_{k,t}$ such that $|s_{k}-\tilde{s}_{k}| < \rho_k \mu_t$. By the mean value theorem, provided $C \geq 1$ is sufficiently large, $|\gamma_{k}(s_{k})-\gamma_{k}(\tilde{s}_{k})|< 2^{-1}C \rho_k \mu_t$. Since $\gamma_k(\tilde{s}_{k})\in M_{k-1,t}\cap W_{k-1}$, it therefore follows that $x \in \cN_{2^{-1}C \rho_k\mu_t}M_{k-1,t}\cap W_{k-1}$. Since this holds for all $k+1\leq t\leq\ell$, we have shown 
     \begin{equation*}
      \gamma_k\big(\Omega_{k,\ell}(\bmu)\big) \subseteq  \gamma_{k}\big(\fN_{k,\ell}(\bmu; \rho_k)\big) \subseteq  \widetilde{\fN}_{k-1,\ell}(\bmu; 2^{-1}C \rho_k),
     \end{equation*}
     which concludes the proof of \eqref{eq: gamma incl}.
\end{proof}

Define $\fN_{k,\ell}(\bmu) := \fN_{k,\ell}(\bmu; 1)$ and $\fN_{k,\ell}^+(\bmu) := \fN_{k,\ell}(\bmu; C_{\circ})$, where $C_{\circ}$ is the constant appearing in the statement of Lemma~\ref{lem: Omega incl}.

 \begin{corollary}\label{cor: chain of supp}
     Let $0\leq k\leq\ell\leq r$, and define the sets 
     \begin{equation*}
        \Omega_{k,\ell, r} (\bmu):=\Phi_{k,\ell}\big( \fN_{\ell,r}(\bmu) \times P_{k,\ell}(\bmu)\big) \quad \textrm{and} \quad \Omega_{k,\ell, r}^+ (\bmu):=\Phi_{k,\ell}\big( \fN_{\ell,r}^+(\bmu) \times P_{k,\ell}(\bmu)\big). 
     \end{equation*}
     Then 
     \begin{equation}\label{eq: chain of supp}
     \fN_{k,r}(\bmu) \subseteq \Omega_{k,\ell, r}^+(\bmu) \qquad \textrm{and} \qquad  \Omega_{k,\ell, r}(\bmu) \subseteq  \fN_{k,r}^+(\bmu)
     \end{equation}
     for all $\bmu = (\mu_1, \dots ,\mu_\ell)$ such that $0 < \mu_1 \leq \cdots \leq \mu_r\leq C_{\circ}\mu_\circ$.
 \end{corollary}
 
 \begin{proof}
     The formula \eqref{eq: phi split} implies that
     \begin{equation*}
         \Omega_{k,r}(\bmu)=\Phi_{k,\ell}\big(\Omega_{\ell,r}(\bmu)\times P_{k,\ell}(\bmu)\big).
     \end{equation*}
     The set inclusions \eqref{eq: chain of supp} then follow from several applications of Lemma~\ref{lem: Omega incl}. In particular,
    \begin{equation*}
    \fN_{k,r}(\bmu) \subseteq \Omega_{k,r}(\bmu) = \Phi_{k,\ell}\big(\Omega_{\ell,r}(\bmu)\times P_{k,\ell}(\bmu)\big) \subseteq \Phi_{k,\ell}\big( \fN_{\ell,r}^+(\bmu) \times P_{k,\ell}(\bmu) \big) = \Omega_{k,\ell, r}^+(\bmu),
    \end{equation*}
    and similarly 
    \begin{equation*}
    \Omega_{k,\ell, r}(\bmu) = \Phi_{k,\ell}\big( \fN_{\ell,r}(\bmu) \times P_{k,\ell}(\bmu) \big) \subseteq  \Phi_{k,\ell}\big(\Omega_{\ell,r}(\bmu)\times P_{k,\ell}(\bmu)\big)  = \Omega_{k,r}(\bmu) \subseteq \fN_{k,r}^+(\bmu),
    \end{equation*}
    as required. 
 \end{proof}




\subsection{Slice formula}\label{subsec: slice formula} 
Continuing with the setup from \S\ref{subsec: parametrising}, let $E_S$ be an extension operator associated to $S$ with amplitude $a \in C^{\infty}_c(U)$ chosen to have support in the open set $W := W_0$. Suppose $f \in L^2(S)$ satisfies $\supp f \subseteq \cN_{\bmu} \cS$ so that, by Lemma~\ref{lem: Omega incl}, we have $\supp (f\cdot a) \subseteq \fN_{0,\ell}(\bmu) \subseteq \Omega_{0,\ell}(\bmu)$ for $1 \leq \ell \leq r$.  Applying the change of variables
\begin{equation*}
  \Phi_\ell:=\Phi_{0,\ell}:W_{\ell}\times  V_{0,\ell} \rightarrow W, 
\end{equation*}
and setting $P_{\ell}(\bmu):=P_{0,\ell}(\bmu)$, we obtain 
\begin{equation}\label{eq: slice formula}
    E_S f(x) = C\int_{P_\ell(\bmu)} e^{i x \cdot \Sigma \circ \Phi_{\ell}(0;\bdeta)} E_{S_\ell(\bdeta)} f_{\ell,\bdeta}(x)\,\ud \bdeta
\end{equation}
where $f_{\ell,\bdeta}(s) := f \circ \Phi_{\ell}(s;\bdeta)$ and $E_{S_\ell(\bdeta)}$ is an extension operator associated to the $d_\ell$-dimensional submanifold $S_\ell(\bdeta) := \Sigma_{\ell,\bdeta}(W_\ell)$  for
\begin{equation*}
    \Sigma_{\ell,\bdeta}(s) := \Sigma\circ\Phi_{\ell}(s;\bdeta) - \Sigma\circ\Phi_{\ell}(0;\bdeta) \quad \textrm{for $s \in W_\ell$ and $\bdeta \in P_\ell(\bmu)$.}
\end{equation*} 
In particular,
\begin{equation}\label{eq: slice extension}
  E_{S_\ell(\bdeta)} g(x) = \int_{W_\ell} e^{i x \cdot \Sigma_{\ell,\bdeta}(s)} g(s) a_{\ell, \bdeta}(s)\,\ud s
\end{equation}
for $a_{\ell, \bdeta}(s) := a\circ\Phi_\ell(s;\bdeta) C^{-1}J_{\bdeta}(s)$, where $J_{\bdeta}(s)$ denotes the Jacobian factor arising from the change of variables and $C \geq 1$ is an admissible constant, chosen to ensure $0 \leq a_{\ell,\bdeta}(s) \leq 1$ for all $s \in W_\ell$. 

By Corollary~\ref{cor: chain of supp}, we have $\supp (f\cdot a) \subseteq \fN_{0,\ell}(\bmu) \subseteq \Omega_{0, \ell, r}^{+}(\bmu)$, and so 
\begin{equation}\label{eq: support slice 1}
    \supp (f_{\ell,\bdeta} \cdot a_{\ell,\bdeta}) \subseteq \fN_{\ell, r}^+(\bmu) \qquad \textrm{for each} \qquad \bdeta \in P_{\ell}(\bmu).
\end{equation}

We refer to \eqref{eq: slice formula} as the \textit{slice formula} for $E_Sf$. We observe the following properties of the slice formula:
\begin{itemize}
    \item $S_\ell(0) = \Sigma \circ \sigma_\ell(W_\ell) \subseteq S_\ell$ is a codimension $0$ submanifold of $S_\ell$. 
    \item Since $f \in L^2(S)$, we have that $f_{\ell,\bdeta} \in L^2(S_\ell)$ for almost every $\bdeta \in P_\ell(\bmu)$. Furthermore, by \eqref{eq: support slice 1}, we have 
\begin{equation}\label{eq: support slice 2}
    \supp  (f_{\ell,\bdeta} \cdot a_{\ell,\bdeta}) \subseteq \cN_{\min\{C_\circ\mu_k,\mu_\circ\}} S_k \quad  \text{ for all $\ell+1 \leq k \leq r$,}
\end{equation}
where the minimum is guaranteed provided $\rho(E_S)$ is sufficiently small. 
\end{itemize}
Here we are using the support (abuse of) notation introduced in \S\ref{subsec: background}. 



\subsection{Local constancy of the slices} Inequalities such as \eqref{eq: local multlinear} are local in the sense that the left-hand $L^{q_k}$-norm is localised to the cube $Q_R := [-R, R]^n$. Consequently, by the uncertainty principle, we should expect our functions to be locally constant at scale $R^{-1}$ in frequency space. In particular, continuing with the setup of the previous subsection, it should be possible to remove the dependence on $\bdeta$ in the $S_\ell(\bdeta)$ when $\bdeta \in P_\ell(\bmu)$ for $\bmu = (\mu_i)_{i=1}^r$ with $0 < \mu_1 \leq \dots \leq \mu_\ell < R^{-1}$. 

To implement the locally constant property rigorously, we go back to our formula for $E_{S_\ell(\bdeta)}$ from \eqref{eq: slice extension} and suppose $\mu_\ell = \max_{i \leq \ell} \mu_{i} < \min\{R^{-1},   C_{\circ}\mu_{\circ}\}$. The phase is given by
\begin{equation*}
    x \cdot \Sigma_{\ell,\bdeta}(s) =  x \cdot \Sigma_{\ell,0}(s) + x \cdot \cE(s;\bdeta)
\end{equation*}
where, by the mean value theorem, $\cE(s,\bdeta)$ satisfies $|\cE(s;\bdeta)| \lesssim |\bdeta|_{\infty} \lesssim R^{-1}$ for all $s \in \supp a_{\ell,\bdeta}$ and all $\bdeta \in P_\ell(\bmu)$. Thus, for $x \in Q_R$, the function $e^{ix \cdot \cE(s;\bdeta)}$ is essentially non-oscillatory and can therefore be removed. More precisely, letting $C \geq 1$ denote a large, admissible constant, by power series expansion
\begin{equation*}
    e^{ix\cdot\cE(s;\bdeta)} = \sum_{\alpha \in \N_0^n} \prod_{j=1}^n \frac{(iCR^{-1}x_j)^{\alpha_j}}{\alpha_j!} (C^{-1}R\cE_j(s;\bdeta))^{\alpha_j}
\end{equation*}
For $x \in Q_R$, we may therefore write\footnote{Here the extension operator $E_{S_\ell}$ is defined as in \eqref{eq: extension def}, with an implicit amplitude $a$ chosen so that \eqref{eq: extension approximation} holds. In particular, we choose this amplitude such that it is identically $1$ on the support of $a_{\ell,\bdeta}$ and with support contained in $W_\ell$.}
\begin{equation}\label{eq: extension approximation}
    E_{S_\ell(\bdeta)}g(x) = \sum_{\alpha \in \N_0^n} B_{\alpha}(x) E_{S_\ell} \big[a_{\bdeta}^{\alpha} g \big](x)
\end{equation}
where, provided $C \geq 1$ is chosen sufficiently large, 
\begin{equation*}
    a_{\bdeta}^{\alpha}(s) := a_{\ell,\bdeta}(s) \prod_{j=1}^n(C^{-1}R\cE_j(s;\bdeta))^{\alpha_j}
\end{equation*} 
satisfies $|a_{\bdeta}^{\alpha}(s)| \leq 1$, and 
\begin{equation*}
    B_{\alpha}(x) := \prod_{j=1}^n\frac{(iCR^{-1}x_j)^{\alpha_j}}{\alpha_j!} \quad \textrm{and} \quad b_{\alpha} := \prod_{j=1}^n \frac{C^{\alpha_j}}{\alpha_j!}
\end{equation*}
satisfy
$|B_{\alpha}(x)| \leq b_{\alpha}$, so that $\sum_{\alpha \in \N_0^n} |B_{\alpha}(t)| \leq e^{nC}$. Thus, by Cauchy--Schwarz,
\begin{equation}\label{eq: loc const}
   |E_{S_\ell(\bdeta)}g(x)|^2 \lesssim  \sum_{\alpha \in \N_0^n} b_{\alpha} |E_{S_\ell} \big[a_{\bdeta}^{\alpha} g\big](x)|^2, \qquad x \in Q_R.
\end{equation}
Since the sequence $b_{\alpha}$ decreases rapidly, \eqref{eq: loc const} effectively bounds $|E_{S_\ell(\bdeta)}g(x)|$ by $|E_{S_\ell}g(x)|$, and therefore is a rigorous interpretation of the locally constant property.




\subsection{Wavepacket decomposition} \label{subsec: wp} Here we construct a variant of the classical wavepacket decomposition adapted to our nested geometry.

\medskip
\noindent \underline{The derivative of $\Phi_{\ell}$.} We first compute the differential $D\Phi_{\ell}|_{(s;\bzero)}$, which is a $ d \times d$ matrix. By first applying the chain rule to the recursive formula \eqref{eq: phi def}, we have 
\begin{equation}\label{eq: Phi differential}
    D\Phi_{k-1,\ell}|_{(s;\bzero_{k-1,\ell})}
    \begin{pmatrix}
        v_\ell \\
        \bw_{k-1,\ell}
    \end{pmatrix}
    =D\gamma_k|_{\gamma_{k,\ell}(s)}D\Phi_{k,\ell}|_{(s;\bzero_{k,\ell})}
    \begin{pmatrix}
        v_\ell \\
        \bw_{k,\ell}
    \end{pmatrix}
    +G_k\big(\gamma_{k,\ell}(s)\big)w_k
\end{equation}
for all $1\leq k\leq \ell\leq r$,  $s \in U_\ell$, $v_{\ell} \in \R^{d_{\ell}}$ and $\bw_{k-1,\ell} = (w_k, \bw_{k,\ell}) \in \R^{m_k} \times \R^{c_\ell - c_k}$. Here we have used \eqref{eq: Phi vs gamma}. 

Let $\partial_s\sigma_{\ell}(s)$ denote the Jacobian matrix of $\sigma_{\ell}$ at $s$ and $B_{k,\ell}(s)\in \mathrm{Mat}(\R, d\times m_{k})$ denote the composition of matrices
\begin{align*}
  B_{k,\ell}(s):&= D\gamma_1|_{\gamma_{1,\ell}(s)} \circ \cdots \circ D\gamma_{k-1}|_{\gamma_{k-1,\ell}(s)} \circ G_k\big(\gamma_{k,\ell}(s)\big), \\
 &=\partial_s\sigma_{k-1}\big(\gamma_{k-1,\ell}(s)\big) G_k\big(\gamma_{k,\ell}(s)\big),
\end{align*}
for $1\leq k\leq\ell$, where here $B_{1,\ell}(s):=G_1\circ\gamma_{1,\ell}(s)$ and the second equality is due to the chain rule. Furthermore, let $\bB_{\ell}(s)\in\mathrm{Mat}(\R, d \times c_\ell)$ denote the block matrix
\begin{equation}\label{eq: normal matrices}
    \bB_{\ell}(s):=
    \begin{bmatrix}
      B_{1, \ell}(s) & B_{2, \ell}(s) & \cdots & B_{\ell,\ell}(s)  
    \end{bmatrix}.
\end{equation}
With these definitions, repeated application of \eqref{eq: Phi differential} yields
\begin{equation}\label{eq: derivative formula}
D\Phi_{\ell}|_{(s;\bzero)}
\begin{pmatrix}
v_\ell \\
\bw_{\ell}    
\end{pmatrix}
= \partial_s\sigma_{\ell}(s)v_\ell+\sum_{k=1}^\ell B_{k,\ell}(s)w_k=\partial_s\sigma_{\ell}(s)v_\ell+\bB_{\ell}(s)\bw_{\ell}
\end{equation}
for $s \in U_{\ell}$, $v_{\ell} \in \R^{d_{\ell}}$ and $\bw_{\ell} \in \R^{d - d_{\ell}}$. Let $b_{j,\ell}(s)$ denote the $j$th column of $\bB_{\ell}(s)$ and $V_{k, \ell}(s) := \mathrm{span}\, \{b_{1,\ell}(s), \dots, b_{c_k,\ell}(s)\}$ for $1 \leq k \leq \ell$. By the non-degeneracy of the parametrisations, 
\begin{equation}
    \bigg|\bigwedge_{j = 1}^{c_\ell}b_{j,\ell}(s)\bigg|\gtrsim 1, \qquad |b_{j,\ell}(s)|\lesssim 1,\qquad \big|V_{k,\ell}(s)\wedge T_{\sigma_{\ell}(s)}M_k\big|\gtrsim 1 \label{eq: b uniform}
\end{equation}
for all $1\leq k \leq \ell\leq r$, $1\leq j\leq c_\ell$, and $s \in U_\ell$.
\medskip

\noindent \underline{Anisotropic decomposition.} We continue with the nested family $\cS = (S, (S_{\ell})_{\ell = 1}^r)$ from \S\ref{subsec: parametrising}. Let $R \geq 1$ and suppose $\bmu = (\mu_1, \dots, \mu_r)$ with $R^{-1} \leq \mu_1 \leq \dots \leq \mu_r \leq \mu_{\circ}$. Note that the upper bound is slightly stronger than that appearing in the setup in previous subsections, whilst the lower bound corresponds to the regime in which we do not have access to the local constancy property \eqref{eq: loc const}. Set 
\begin{equation*}
    \cL_{R^{1/2},\bmu} := \big\{1 \leq \ell \leq r : \mu_{\ell} > R^{-1/2} \big\} \quad \textrm{and} 
    \quad \ell_{R^{1/2},\bmu} := 
    \begin{cases}
        \min \cL_{R^{1/2},\bmu} - 1 & \textrm{if $\cL_{R^{1/2}} \neq \emptyset$,} \\
        r & \textrm{otherwise},
    \end{cases}
\end{equation*}
so that if $\cL_{R^{1/2},\bmu} \neq \emptyset$, then $\cL_{R^{1/2},\bmu} = \{\ell_{R^{1/2},\bmu} + 1, \dots, r\}$. For the ease of notation in this section, we set $\ell_{\star} :=\ell_{R^{1/2},\bmu}$. We define a $d \times d$ diagonal matrix
\begin{equation*}
D_{R^{1/2},\bmu} := 
\begin{bmatrix}
     \mu_1 I_{m_{1}}  & \cdots & \bzero & \bzero \\
     \vdots & \ddots & \vdots & \vdots \\
     \bzero &  \cdots & \mu_{\ell_{\star}} I_{m_{\ell_{\star}}} &  \bzero \\
     \bzero &  \cdots & \bzero &  R^{-1/2} I_{d_{\star}}
\end{bmatrix},
\end{equation*}
where $d_{\star} := d - \sum_{\ell=1}^{\ell_{\star}}m_{\ell}= \dim(S_{\ell_{\star}})$.  
Finally, for $\bB_\ell(s)$ as in \eqref{eq: normal matrices}, define the $d \times d$ matrices
\begin{equation}\label{eq: Lambda def 1}
   \Lambda_{\ell}(s) := \begin{bmatrix}
        \bB_\ell(s) & \displaystyle\frac{\partial \sigma_{\ell}}{\partial s}(s) 
    \end{bmatrix}, \qquad s \in W_{\ell},
\end{equation}
which, by \eqref{eq: b uniform}, are invertible for $1 \leq \ell \leq r$, and let
\begin{equation}\label{eq: Lambda def 2}
   \Lambda_{R, \bmu}(s) := 
   \begin{cases}
      C_\circ^{1/2}\Lambda_{\ell_{\star}}(s) \circ
    D_{R^{1/2},\bmu} & \textrm{if $1 \leq \ell_{\star} \leq r$,} \\
    C_\circ^{1/2}R^{-1/2}I_d & \textrm{if $\ell_{\star} = 0$,}
   \end{cases}
\end{equation}
noting that $D_{R^{1/2},\bmu} = R^{-1/2}I_d$ in the $\ell_{\star} = 0$ case. 

Define a grid of points
\begin{equation*}
    \Gamma_{R,\bmu}(\cS) := R^{-1/2}\Z^{d_{\star}} \cap  \fN_{\ell_{\star}, r}(\bmu; 2C_{\circ})
\end{equation*}
where $\fN_{\ell_{\star}, r}(\bmu; 2C_{\circ})$ is as defined in \eqref{eq: neighbourhood notation} and $C_\circ \geq 1$ is an admissible constant, chosen large enough so that the conclusion of Lemma~\ref{lem: Omega incl} holds. We form a covering of
\begin{equation*}
  \fN_r(\bmu) := \fN_{0, r}(\bmu) = \cN_{\mu_1}M_1 \cap \cdots \cap \cN_{\mu_r}M_r \cap W  
\end{equation*}
by parallelepipeds. Let $\Theta_{R,\bmu}(\cS)$ denote the family of all parallelepipeds of the form
\begin{equation}\label{eq: Lambda def 3}
    \theta = u_{\theta} + \Lambda_{\theta}([-1,1]^{d}) \quad \textrm{where $u_{\theta} := \sigma_{\ell_{\star}}(s_{\theta})$ and $\Lambda_{\theta} := \Lambda_{R, \bmu}(s_{\theta})$}
\end{equation}
for $s_{\theta} \in \Gamma_{R,\bmu}(\cS)$. Here $\sigma_{0}:=\mathrm{Id}$.

\begin{lemma}\label{lem: wp} Provided $C_{\circ} \geq 1$ is sufficiently large, the following properties hold:
\begin{enumerate}[i)]
    \item The collection $\Theta_{R,\bmu}(\cS)$ forms a cover of $\fN_r(\bmu)$.
    \item The scaled parallelepipeds $\{4 \cdot \theta : \theta \in \Theta_{R,\bmu}(\cS)\}$ are finitely-overlapping. 
    \item $4 \cdot \theta \cap W \subseteq \fN_r(C_{\circ}^2 \bmu)$ for all $\theta \in \Theta_{R,\bmu}(\cS)$.
\end{enumerate}
\end{lemma}

\begin{proof}
If $\ell_{\star}=0$, then the $\Theta_{R,\bmu}(\cS)$ is a standard covering by $R^{1/2}$ cubes, and the desired properties are immediate. We therefore assume $1 \leq \ell_{\star} \leq r$.\medskip

\noindent i) By Corollary~\ref{cor: chain of supp}, it suffices to show that $\Theta_{R,\bmu}(\cS)$ forms a cover of $\Omega_{0, \ell_{\star}, r}^+(\bmu)$. Fix $x \in \Omega_{0, \ell_{\star}, r}^+(\bmu)$ so that 
\begin{equation*}
  x = \Phi_{0, \ell_{\star}}(s; \bdeta) = \Phi_{\ell_{\star}}(s; \bdeta) \quad \textrm{for some $s \in \fN_{\ell_{\star}, r}^+(\bmu)$ and $\bdeta \in P_{0,\ell_{\star}}(\bmu) = P_{\ell_{\star}}(\bmu)$.}
\end{equation*}
Since $\fN_{\ell_{\star}, r}^+(\bmu)=\fN_{\ell_{\star}, r}(\bmu; C_\circ)$,  there exists some $s_{\theta} \in \Gamma_{R,\bmu}(\cS)$ such that $|s - s_{\theta}|_{\infty} < R^{-1/2}$. The problem is therefore reduced to showing $\Phi_{\ell_{\star}}(s;\bdeta) \in \theta$, provided $C_{\circ} \geq 1$  is sufficiently large.

For $u_{\theta} := \sigma_{\ell_{\star}}(s_{\theta}) =\Phi_{\ell_{\star}}(s_\theta; \mathbf{0})$, write
\begin{equation*}
 \Phi_{\ell_{\star}}(s;\bdeta) - u_{\theta} = \bB_{\ell_{\star}}(s_{\theta})\bdeta + \partial_s\sigma_{\ell_{\star}}(s_{\theta})(s-s_{\theta}) +  \cE_{\theta}(s;\bdeta) 
 \end{equation*}
 where
\begin{equation}\label{eq: error}
 \cE_{\theta}(s;\bdeta) :=    \Phi_{\ell_\star}(s;\bdeta) -  \sigma_{\ell_{\star}}(s_{\theta}) - \partial_s\sigma_{\ell_{\star}}(s_{\theta})(s-s_{\theta}) - \bB_{\ell_{\star}}(s_{\theta})\bdeta.
\end{equation}
By combining the definitions \eqref{eq: Lambda def 1}, \eqref{eq: Lambda def 2} and \eqref{eq: Lambda def 3}, we have
\begin{equation*}
 \Phi_{\ell_{\star}}(s;\bdeta) - u_{\theta} = \Lambda_{\theta}(\bxi + \bzeta) \qquad \textrm{where} \qquad \bzeta := \big(\Lambda_{\theta}\big)^{-1}\cE_{\theta}(s;\bdeta)
\end{equation*}
and 
\begin{equation*}
    \bxi := C_\circ^{-1/2}\big(D_{R^{1/2}, \bmu}\big)^{-1}
    \begin{bmatrix}
        \bdeta \\
        s - s_{\theta}
    \end{bmatrix} \in [-1/2,1/2]^d;
\end{equation*}
here we are assuming that $C_{\circ}^{3/8} \geq 2$ and using the definition of $P_{\ell_{\star}}(\bmu)$. 
 Since $\mu_k \leq R^{-1/2}$ for $1 \leq k \leq \ell_{\star}$, by \eqref{eq: derivative formula}, \eqref{eq: Phi vs gamma} and Taylor's theorem, we have
\begin{equation}\label{eq: error est}
    |\cE_{\theta}(s;\bdeta)| \lesssim |s - s_{\theta}|_{\infty}^2 + |\bdeta|_{\infty}^2 \lesssim C_{\circ}^{1/4}R^{-1}.
\end{equation}
Thus, since we also have $\mu_k \geq R^{-1}$ for $1 \leq k \leq \ell_{\star}$,  we can ensure that $\bzeta \in [-1/2,1/2]^d$ provided that $C_{\circ}\geq 1$ is sufficiently large.

Combining the above observations, $\Phi_{\ell_{\star}}(s;\bdeta) = u_{\theta}+  \Lambda_{\theta}(\bxi + \bzeta)$ where $\bxi + \bzeta \in [-1,1]^d$ and so $\Phi_{\ell_{\star}}(s;\bdeta) \in \theta$, as required. \medskip

\noindent ii) The finite overlap property is immediate. Indeed, if $\theta_1$, $\theta_2 \in \Theta_{R, \bmu}(\cS)$ are such that $4 \cdot \theta_1 \cap 4 \cdot \theta_2 \neq \emptyset$, then it follows that the corresponding centres $u_{\theta_i} = \sigma_{\ell_{\star}}(s_{\theta_i})$ satisfy $|u_{\theta_1} - u_{\theta_2}|_{\infty} \lesssim R^{-1/2}$. Since $\sigma_{\ell_{\star}}$ maps diffeomorphically onto its image, we conclude that $|s_{\theta_1} - s_{\theta_2}|_{\infty} \lesssim R^{-1/2}$. But for $\theta_1$ fixed, this is only possible for $O(1)$ choices of $\theta_2$ since the $s_{\theta_2}$ lie on $R^{-1/2}$ separated points. \medskip 

\noindent iii) Fix $\theta \in \Theta_{R,\bmu}(\cS)$ with centre $u_{\theta} = \sigma_{\ell_{\star}}(s_{\theta})$ for some $s_{\theta} \in \Gamma_{R,\bmu}(\cS)$.

Let $\ell_{\star} + 1 \leq t \leq r$ so that $s_{\theta} \in \cN_{2C_\circ\mu_t} M_{\ell_{\star},t}$. Thus there exists some $s_{\theta,t} \in M_{\ell_{\star},t}$ with $|s_{\theta} - s_{\theta,t}| < 2C_\circ \mu_t$. Consequently, $|u_{\theta} - u_{\theta,t}| \lesssim C_\circ\mu_t$ where $u_{\theta,t} := \sigma_{\ell_{\star}}(s_{\theta,t}) \in M_t$ and so $\dist(u_{\theta}, M_t) \lesssim C_\circ\mu_t$. On the other hand, it is clear from the definition that $4 \cdot \theta$ lies in a ball of radius $O(C_\circ^{1/2} R^{-1/2})$ centred at $u_{\theta}$. Since $\mu_t > R^{-1/2}$, we conclude that 
\begin{equation}\label{eq: cap containment 1}
    4 \cdot \theta \subseteq \bigcap_{t = \ell_{\star} + 1}^r \cN_{C_{\circ}^2 \mu_t} M_t,
\end{equation}
provided $C_{\circ} \geq 1$ is chosen sufficiently large.

We now consider $1 \leq k \leq \ell_{\star}$. Suppose  $u \in  4 \cdot \theta  \cap W$, so that, by the definitions \eqref{eq: Lambda def 1}, \eqref{eq: Lambda def 2} and \eqref{eq: Lambda def 3}, we can write 
\begin{equation*}
    u = \bB_{\ell_{\star}}(s_{\theta}) \bdeta + \partial_s \sigma_{\ell_{\star}}(s_{\theta})(s- s_{\theta}) + \sigma_{\ell_{\star}}(s_{\theta})
\end{equation*}
for some $\bdeta \in P_{\ell_{\star}}(4C_{\circ}^{3/8} \bmu)$ and $s - s_{\theta} \in [-4C_{\circ}^{1/2}R^{-1/2}, 4C_{\circ}^{1/2}R^{-1/2}]^{d_{\star}}$.  Thus,
\begin{equation*}
    u = \Phi_{\ell_{\star}}(s; \bdeta) - \cE_{\theta}(s;\bdeta),
\end{equation*}
where $\cE_{\theta}(s;\bdeta)$ is as defined in \eqref{eq: error}. Arguing as in \eqref{eq: error est}, we have $|\cE_{\theta}(s;\bdeta)| \lesssim C_\circ R^{-1}$. 

Since $\Phi_{\ell_{\star}} \colon W_{\ell_{\star}} \times V_{\ell_{\star}} \to W$ is a bijection, we can write $u = \Phi_{\ell_{\star}}(\tilde{s}; \tilde{\bdeta})$ for some $\tilde{s} \in W_{\ell_{\star}}$ and $\tilde{\bdeta} \in  V_{ \ell_{\star}} := V_{0, \ell_{\star}}$. Taking inverses, 
\begin{equation*}
    (\tilde{s}-s; \tilde{\bdeta}-\bdeta) = \Phi_{\ell_{\star}}^{-1}\big(\Phi_{\ell_{\star}}(s; \bdeta) - \cE_{\theta}(s;\bdeta)\big) - \Phi_{\ell_{\star}}^{-1}\big(\Phi_{\ell_{\star}}(s; \bdeta)\big),
\end{equation*}
and therefore, by the mean value theorem, $|\tilde{\bdeta} -  \bdeta|_{\infty} \lesssim C_\circ R^{-1}$. Since $\mu_k \geq R^{-1}$ for $1 \leq k \leq r$  and $\bdeta \in P_{\ell_{\star}}(4C_{\circ}^{3/8} \bmu)$, we have $\tilde{\bdeta} \in P_{\ell_{\star}}(C_{\circ}\bmu)$, provided $C_{\circ} \geq 1$ is chosen sufficiently large. By the definition \eqref{eq: Omega def}, we have $u = \Phi_{\ell_\star}(\tilde{s},\tilde{\bdeta}) \in \Omega_{0,\ell_\star}(C_\circ\bmu)$. Our hypotheses ensure that $C_{\circ} \mu_1 \leq \dots \leq C_{\circ}\mu_r \leq C_{\circ}\mu_{\circ}$,  and so we may apply Lemma~\ref{lem: Omega incl} to deduce that $  u \in  \fN_{\ell_{\star}}^+(C_{\circ} \bmu) =  \fN_{\ell_{\star}}(C_{\circ}^2 \bmu)$.  Thus, 
\begin{equation}\label{eq: cap containment 2}
  4 \cdot \theta \cap W \subseteq \fN_{\ell_{\star}}(C_{\circ}^2 \bmu)
\end{equation}
 and combining \eqref{eq: cap containment 1} and \eqref{eq: cap containment 2} concludes the proof.
\end{proof}

\noindent \underline{Definition of the wave packets.}  Let $\psi \in C^{\infty}_c(\R^d)$ satisfy $0 \leq \psi \leq 1$; $\psi(u) = 1$ if $|u|_{\infty} \leq 1$ and $\psi(u) = 0$ if $|u|_{\infty} \geq 2$. Define $\tilde{\psi}(u) := \psi(u/2)$, so that $\psi \tilde{\psi} = \psi$. Given $\theta \in \Theta_{R,\bmu}(\cS)$, let 
\begin{equation*}
    \psi_{\theta}(u) := \psi\big(\Lambda_{\theta}^{-1}(u - u_{\theta})\big)  \quad \textrm{and} \quad  \tilde{\psi}_{\theta}(u) := \tilde{\psi}\big(\Lambda_{\theta}^{-1}(u - u_{\theta})\big).
\end{equation*}
Thus, $\psi_{\theta}(u) = 1$ if $u \in \theta$, $\supp \psi_{\theta} \subseteq 2 \cdot \theta$, $\supp \tilde{\psi}_{\theta} \subseteq 4 \cdot \theta$ and $\psi_{\theta}\tilde{\psi}_{\theta} = \psi_{\theta}$. By Lemma~\ref{lem: wp} i) and ii), we have
\begin{equation*}
 1 \leq \sum_{\theta \in \Theta_{R,\bmu}(\cS)} \psi_{\theta}(u) \lesssim 1  \qquad \textrm{for all $u \in \fN_r(\bmu)$.}   
\end{equation*}

Now suppose $f \in L^2(S)$ is smooth and satisfies $\supp f \subseteq \cN_{\bmu} \cS$. This abuse of notation translates as $f \in C^{\infty}_c(U)$ with $\supp f \subseteq \fN_r(\bmu)$, provided we also assume $\supp f \subseteq W$. Under these conditions, we may write 
 \begin{equation*}
  f = \sum_{\theta \in \Theta_{R,\bmu}(\cS)} f_{\theta} \psi_{\theta} \qquad \textrm{where $f_{\theta} \in C^{\infty}_c(U)$ satisfy $|f_{\theta}(u)| \lesssim |f(u)|$ for all $u \in U$.}   
 \end{equation*}

Letting $\Lambda_{\theta}^{-\top}(\Z^d) := \{ \Lambda_{\theta}^{-\top} m : m \in \Z^d \}$, define
\begin{equation*}
    \cT_{R,\bmu}(\cS) := \Theta_{R,\bmu}(\cS) \times \Lambda_{\theta}^{-\top}(\Z^d) \quad \textrm{and} \quad \cT_{R,\bmu}(\cS;\theta) := \{\theta\} \times \Lambda_{\theta}^{-\top}(\Z^d)
\end{equation*}
for all $\theta \in \Theta_{R,\bmu}(\cS)$. By rescaling and applying a Fourier series decomposition,
\begin{equation*}
   f(u) = \sum_{(\theta, v) \in \cT_{R,\bmu}(\cS)} f_{\theta, v}(u)  \quad \textrm{and} \quad (f_{\theta} \psi_{\theta})(u) = \sum_{(\theta, v) \in \cT_{R,\bmu}(\cS;\theta)} f_{\theta, v}(u)
\end{equation*}
for each $\theta \in \Theta_{R,\bmu}(\cS)$, where
\begin{equation*}
    f_{\theta,v}(u) := (2\pi)^{-d}|\det \Lambda_{\theta}|^{-1}   (f_{\theta}  \psi_{\theta})\;\widehat{}\;(v) e^{i v \cdot u} \tilde{\psi}_{\theta}(u).
\end{equation*}
Here $(f_{\theta} \psi_{\theta})\;\widehat{}\;(v) = \int_{\R^d} e^{-i v \cdot u} (f_{\theta} \psi_{\theta})(u)\,\ud u$ is the Fourier transform of $f_{\theta} \psi_{\theta}$ evaluated at $v$.
\medskip

\noindent \underline{Basic properties of the wave packets.} We identify three key properties of the above decomposition:\medskip

\noindent\textit{Orthogonality.} Parseval's identity for Fourier series and Lemma~\ref{lem: wp} ii) give
\begin{equation*}
       \big\| \sum_{(\theta, v) \in \cT} f_{\theta, v} \big\|_{L^2(\R^d)}^2 \lesssim \sum_{(\theta, v) \in \cT} \|f_{\theta, v}\|_{L^2(\R^d)}^2 \lesssim \|f\|_{L^2(\R^d)}^2
    \end{equation*}
for any collection $\cT \subseteq \cT_{R,\bmu}(\cS)$.\medskip

\noindent\textit{Spatial localisation.} Recall $\Sigma_{\ell_{\star}} := \Sigma \circ \sigma_{\ell_{\star}}$ and $S_{\ell_{\star}} = \Sigma_{\ell_{\star}}(U_{\ell_{\star}})$. Given $\varepsilon > 0$, let $\varepsilon_{\circ} := \varepsilon/(100n)$ and define the slabs
\begin{equation}\label{eq: slabs 1}
    T_{\theta,v} := \big\{x \in \R^n : \big|(\partial_s\Sigma_{\ell_{\star}})(s_{\theta})^{\top}x + v_{\theta}^{\star}\big|_{\infty} < R^{1/2 + \varepsilon_{\circ}} \big\},  \quad v_{\theta}^{\star} := (\partial_s\sigma_{\ell_{\star}})(s_{\theta})^{\top}v
\end{equation}
for all $(\theta, v) \in \cT_{R,\bmu}(\cS)$. Let $u_{\theta}$ be as defined in \eqref{eq: Lambda def 3}. Observe that, by the chain rule formula 
\begin{equation*}
  (\partial_s\Sigma_{\ell_{\star}})(s_{\theta}) = (\partial_u\Sigma)(u_{\theta}) \circ (\partial_s\sigma_{\ell_{\star}})(s_{\theta})  
\end{equation*} 
and the definition of $\Lambda_{\theta}$, we have $\widetilde{T}_{\theta,v} \subseteq T_{\theta,v}$ where 
\begin{equation}\label{eq: slabs 2}
   \widetilde{T}_{\theta,v} := \big\{x \in Q_R : \big|\Lambda_{\theta}^{\top}\big((\partial_u\Sigma)(u_{\theta})^{\top}x +  v\big)\big|_{\infty} < R^{\varepsilon_{\circ}} \big\}.
\end{equation}

Let $\bT_{R,\bmu}(\cS) := \{T_{\theta, v} : (\theta, v) \in \cT_{R,\bmu}(\cS)\}$ denote the collection of all slabs of the form \eqref{eq: slabs 1} and, for each $T = T_{\theta, v} \in \bT_{R,\bmu}(\cS)$, write $f_T := f_{\theta,v}$. A standard integration-by-parts argument then shows that
\begin{equation*}
    |E_Sf_T(x)| \lesssim_{N, \varepsilon}   |E_Sf_T(x)| \chi_T(x) +  R^{-N}\|f\|_{L^2(U)}
\end{equation*}
for all $x \in Q_R$ and all $N \in \N_0$, where $E_S$ has amplitude $a$ supported in the open set $W$. Indeed, we have a stronger estimate with $\chi_{T}$ replaced with $\chi_{\widetilde{T}}$ for $\widetilde{T}$ as defined in \eqref{eq: slabs 2}, but, for our purposes, we only require localisation to the larger slab $T$.  \medskip

\noindent\textit{Preservation of support.} With the above setup, for $C_{\circ}^2\bmu = (C_{\circ}^2\mu_{\ell})_{\ell = 1}^r$, we have
\begin{equation}\label{eq: preserve supp}
 \supp f_T \cap W \subseteq \fN_r(C_{\circ}^2\bmu) \qquad \textrm{for all $T \in \bT_{R,\bmu}(\cS)$.}
\end{equation}
 Indeed, this follows from Lemma \ref{lem: wp} iii) after unravelling the notation and noting that $\supp \tilde{\psi}_{\theta} \subseteq 4 \cdot \theta$. The localisation \eqref{eq: preserve supp} is the key advantage of the above \textit{anisotropic} wave packet decomposition, as opposed to the regular wave packet decomposition used, for instance, in \cite{BCT2006, BBFL2018}, and is crucial for the forthcoming induction-on-scales argument.




\section{The recursive scheme}

Throughout this section, we fix $2 \leq k \leq n$ and 
\begin{equation*}
    \sE = ((\cS_j)_{j = 1}^k, \bq) \quad \textrm{ with } \quad \cS_j = (S_j, (S_{j,\ell})_{\ell = 1}^{r_j}) 
\end{equation*}
a transverse ensemble in $\R^n$ of codimension $(m(j,\ell))_{j,\ell}$.  We apply the observations of \S\ref{sec: extension} to each of the nested families of submanifolds $\cS_j$. In particular, for $1 \leq j \leq k$ we let $\Phi_{k,\ell}^{(j)} \colon W_{\ell}^{(j)} \times V_{k,\ell}^{(j)} \to W_k^{(j)}$ denote the diffeomorphism chain and $\mu_{\circ}^{(j)} > 0$ be the threshold width for $\cS_j$, as defined in \eqref{eq: diff chain} and on page \pageref{page: threshold}, respectively. We let $\mu_{\circ} := \min\{\mu_{\circ}^{(1)}, \dots, \mu_{\circ}^{(k)}\} > 0$.




\subsection{A reformulation}\label{subsec: reform}  Given $R \geq 1$ and $\bmu = (\bmu_j)_{j = 1}^k$ with $\bmu_j := (\mu_{j,\ell})_{\ell = 1}^{r_j}$ a compatible family of scales for $\sE$, set
\begin{equation*}
    \cL_{R, \bmu}(j) := \big\{1 \leq \ell \leq r_j : \mu_{j,\ell} > R^{-1} \big\}
\end{equation*}
and
\begin{equation*}
    \ell_{R, \bmu}(j) := 
    \begin{cases}
        \min \cL_{R, \bmu}(j) - 1 & \textrm{if $\cL_{R, \bmu}(j) \neq \emptyset$,} \\
         r_j & \textrm{otherwise},
    \end{cases}
\end{equation*}
so that if $\cL_{R, \bmu}(j) \neq \emptyset$, then $\cL_{R, \bmu}(j) = \{\ell_{R, \bmu}(j) + 1 , \dots, r_j\}$. We then set
\begin{equation}\label{eq: recursive manifolds}
     S_{j, R} := S_{j, \ell_{R, \bmu}(j)} \quad \textrm{and} \quad U_{j, R} := U_{j, \ell_{R, \bmu}(j)}, \quad 1 \leq j \leq k,
\end{equation}
where $S_{j, 0} := S_j$ and $U_{j,0} := U_j$; note that $S_{j,R}$ and $U_{j,R}$ also depend on $\bmu$, but we suppress this for notational simplicity. Consider also the restricted nested family of submanifolds $\cS_{j,R}:=(S_{j,R}, (S_{j,\ell})_{\ell \in \cL_{R,\bmu}(j)})$ for $1 \leq j \leq k$.

We let $\rho_{\circ} > 0$ be a constant, chosen sufficiently small for the purposes of the forthcoming argument. In particular, we assume $B(0, \rho_{\circ}) \subseteq W_0^{(j)}$ for $1 \leq j \leq k$. For all $R \geq 1$ and $\bmu$ as above, let $\fR(R; \bmu)$ denote the smallest constant $C \geq 1$ for which the inequality 
\begin{equation}\label{eq: desired}
    \int_{Q_R} \prod_{j = 1}^k |E_{S_{j, R}}f_j|^{q_j}  \leq C \bC_R(\bmu) \prod_{j = 1}^k \|f_j\|_{L^2(U_{j, R})}^{q_j}
\end{equation}
holds for
\begin{equation}\label{eq: recursive const}
    \bC_R(\bmu) := \prod_{j = 1}^k \prod_{\ell \in \cL_{R, \bmu}(j)} (\mu_{j,\ell})^{m(j,\ell)q_j/2}
\end{equation}
whenever $\rho(E_{S_{j,R}}) < \rho_{\circ}^{n - \ell_{R, \bmu}(j)}$ for $1 \leq j \leq k$ and the $f_j \in L^2(S_{j,R})$ are smooth, compactly supported and satisfy\footnote{Since $f_j \in L^2(S_{j,R})$, the reference manifolds are $S_{j,R}$, and we understand $\cN_{\mu_{j,\ell}}(S_{j,\ell})$ as the  neighbourhood in $\R^{\dim(S_{j,R})}$ of the submanifold $\gamma_{j,\ell_{R,\bmu(j)} +1} \circ \cdots \circ \gamma_{j,\ell}(U_{j,\ell})$.}
\begin{equation*}
     \supp f_j \subseteq \cN_{\mu_{j,\ell}} S_{j,\ell}\quad \textrm{for all $\ell \in \cL_{R, \bmu}(j)$.}
\end{equation*}
If $\cL_{R, \bmu}(j) = \emptyset$, then the product over $\ell \in \cL_{R, \bmu}(j)$ in \eqref{eq: recursive const} is interpreted as equal to $1$. We also define
\begin{equation*}
    \fR(R) := \sup_{\bmu,\, \mu_{j,\ell} \leq \mu_{\circ}} \fR(R;\bmu),
\end{equation*}
where the supremum is taken over all compatible families of scales $\bmu=(\bmu_j)_{j=1}^k$ for $\sE$ with $\bmu_j=(\mu_{j,\ell})_{\ell=1}^{r_j}$ satisfying $0<\mu_{j,\ell} \leq \mu_\circ$ for $1 \leq \ell \leq r_j, 1 \leq j \leq k$. 

Our goal is to show that, for all $\varepsilon > 0$, the bound $\fR(R) \lesssim_{\varepsilon} R^{\varepsilon}$ holds for all $R \geq 1$. From this we shall deduce Theorem~\ref{thm: nested}: see \S\ref{subsec: recursive step}. We further reformulate our goal as follows. Let $\bmu$ be a compatible family of scales for $\sE$. We first observe a trivial estimate for $\fR(R;\mu)$, by bounding the left-hand side of \eqref{eq: desired} in terms of the $L^{\infty}$-norms of the $|E_{S_{j, R}}f_j|$ and applying a Riemann--Lebesgue-type argument to arrive at an expression involving the $L^1$ norms of the $f_j$. We can then pass to $L^2$ norms using Cauchy--Schwarz, taking advantage of the localisation of the $\supp f_j$, to deduce that
\begin{equation}\label{eq: poly size}
   \fR(R; \bmu) \lesssim R^n \qquad \textrm{for all $\bmu$ compatible for $\sE$, and so} \qquad \fR(R) \lesssim R^n.
\end{equation}
Thus, the \textit{upper exponent}
\begin{equation}\label{eq: upper exponent}
    \eta_{\mathrm{exp}} := \limsup_{R \to \infty} \frac{\log \fR(R)}{\log R}
\end{equation}
is a well defined real number which satisfies $\eta_{\mathrm{exp}} \leq n$. To prove Theorem~\ref{thm: nested}, it shall suffice to show $\eta_{\mathrm{exp}} \leq 0$.




\subsection{Uncertainty principle reduction}\label{subsec: uncertainty red} As a manifestation of the uncertainty principle, $\fR(R; \bmu)$ automatically bounds a wider variety of multilinear expressions.  In what follows, we let $C_\circ \geq 1$ be an admissible constant, chosen sufficiently large for the purposes of the forthcoming argument and so that the conclusions of \S\ref{sec: extension} hold. Let $\bmu = (\bmu_j)_{j = 1}^k$ be a compatible family of scales for $\sE$, with $\bmu_j = (\mu_{j,\ell})_{\ell = 1}^{r_j}$  and $0 < \mu_{j,\ell} \leq \mu_{\circ}$ for $1 \leq \ell \leq r_j$, $1 \leq j \leq k$. We define $\btmu = (\btmu_j)_{j = 1}^k$ with $\btmu_j = (\tilde{\mu}_{j,\ell})_{\ell = 1}^{r_j}$ for $\tilde{\mu}_{j,\ell} := \min\{C_{\circ}^2\mu_{j,\ell}, \mu_{\circ}\}$, $1 \leq \ell \leq r_j$, $1 \leq j \leq k$.

\begin{lemma}\label{lem: uncer red} Let $R \geq 1$ and $\cL_{R, \btmu}(j) \subseteq \cL^{\star}(j) \subseteq \{1, \dots, r_j\}$ for $1 \leq j \leq k$. Suppose that whenever $\ell \in \cL^{\star}(j)$ and $\ell \leq \ell' \leq r_j$, we have $\ell' \in \cL^{\star}(j)$. Define
\begin{equation*}
    \ell^{\star}(j) :=
    \begin{cases}
        \min \cL^{\star}(j) - 1 & \textrm{if $\cL^{\star}(j) \neq \emptyset$,} \\
        r_j & \textrm{otherwise,}
    \end{cases}
\end{equation*}
so that if $\cL^{\star}(j) \neq \emptyset$, then $\cL^{\star}(j) := \{\ell^{\star}(j) + 1, \dots, r_j\}$. Set
\begin{equation*}
    S_j^{\star} := S_{j, \ell^{\star}(j)}
    \quad \textrm{and} \quad 
    U_j^{\star} := 
        U_{j, \ell^{\star}(j)} \quad \textrm{for $1 \leq j \leq k$.}
\end{equation*}
If $\rho(E_{S_j^{\star}}) < \rho_{\circ}^{n - \ell^{\star}(j)}$ for $1 \leq j \leq k$, then 
\begin{equation}\label{eq: uncer red}
   \int_{Q_R} \prod_{j = 1}^k |E_{S_j^{\star}}f_j|^{q_j} \lesssim  \fR(R;  \btmu) \prod_{j = 1}^k \prod_{\ell \in \cL^{\star}(j)} (\mu_{j,\ell})^{m(j,\ell) q_j/2}\prod_{j = 1}^k \|f_j\|_{L^2(U_j^{\star})}^{q_j}
\end{equation}
holds whenever the $f_j \in L^2(S_j^{\star})$ are smooth, compactly supported and 
\begin{equation}\label{eq: uncer red supp}
     supp f_j \subseteq \cN_{\mu_{j,\ell}} S_{j,\ell} \quad \textrm{for all $\ell \in \cL^{\star}(j)$.}
\end{equation}
\end{lemma}

\begin{proof} Let $\cJ_R := \{1 \leq j \leq k : \ell^{\star}(j) = \ell_{R,  \btmu}(j)\}$ and $\cJ^{\star} := \{1, \dots, k\}\setminus \cJ_R$, so that $S_j^{\star} = S_{j, R}$ for all $j \in \cJ_R$ and $\ell^{\star}(j) < \ell_{R, \btmu}(j)$ for $j \in \cJ^{\star}$. Provided $\rho_\circ>0$ is chosen sufficiently small, we can apply the slice formula \eqref{eq: slice formula} to the pairs $S_j^{\star}$ and $S_{j,R}(\bdeta)$ for $j \in \cJ^{\star}$ and we have, by Cauchy--Schwarz, 
 \begin{equation}\label{eq: uncer 1}
\int_{Q_R}\prod_{j = 1}^k |E_{S_j^{\star}}f_j|^{q_j} \lesssim  \prod_{j \in \cJ^{\star}} \prod_{\ell \in \fm_j} (\mu_{j, \ell})^{m(j, \ell)q_j/2} I(R;\bmu)
\end{equation}
where $\fm_j := \{\ell^{\star}(j) + 1, \dots, \ell_{R,  \btmu}(j) \}$ and
\begin{equation*}
  I(R;\bmu) := \int_{Q_R} \prod_{j  \in \cJ_R} |E_{S_{j,R}}f_j(x)|^{q_j} \prod_{j \in \cJ^{\star}} \Big(\int_{P_j(\bmu_j)}|E_{S_{j,R}(\bdeta)}f_{j, \bdeta}(x)|^2 \,\ud \bdeta\Big)^{q_j/2}\,\ud x, 
\end{equation*}
for 
\begin{equation*}
   P_j(\bmu_j) := \prod_{\ell \in \fm_j} [-C_{\circ}^{1/8}\mu_{j,\ell}, C_{\circ}^{1/8}\mu_{j,\ell}]^{m(j,\ell)} \qquad \textrm{for $j \in \cJ^{\star}$.} 
\end{equation*}
Here, for an appropriate map $\Phi_j^{\star}$ satisfying $\Phi_j^{\star}(0;\bzero) = 0$,\footnote{Explicitly, $\Phi_j^{\star} := \Phi_{\ell^{\star}(j),\ell_{R, \btmu}(j)}^{(j)} \colon W^{(j)}_{\ell_{R, \btmu}(j)} \times V_{\ell^{\star}(j),\ell_{R, \btmu}(j)}^{(j)} \to W_{\ell^{\star}(j)}^{(j)}$.}  we have
\begin{equation*}
    f_{j, \bdeta}(s) := f_j\circ \Phi_j^{\star}(s;\bdeta) \qquad \textrm{for $j \in \cJ^{\star}$.}
\end{equation*}
Observe that $\cL^{\star}(j)  = \cL_{R, \btmu}(j)  \cup \fm_j$ for $1 \leq j \leq k$, where the union is disjoint and $\fm_j = \emptyset$ whenever $j \in \cJ_R$. 

Fix $j \in \cJ^{\star}$ so that $\ell^{\star}(j) < \ell_{R, \btmu}(j)$. By the definition and the ordering of the $(\mu_{j,\ell})_{\ell = 1}^{r_j}$, we have
\begin{equation*}
  \{ 1 \leq \ell \leq r_j :  \tilde{\mu}_{j,\ell} \leq R^{-1}\} = \{1, \dots, \ell_{R,  \btmu}(j)\} 
\end{equation*}
and so $\mu_{j,\ell} \leq  \tilde{\mu}_{j, \ell} \leq R^{-1}$ for all $\ell \in \fm_j$. Consequently, given $g \in C^{\infty}_c(U_{j,R})$, from \eqref{eq: loc const} we have 
\begin{equation}\label{eq: applying slice}
   |E_{S_{j,R}(\bdeta)}g(x)|^2 \lesssim  \sum_{\alpha[j] \in \N_0^n} b_{\alpha[j]} |E_{S_{j,R}} \big[a_{j,\bdeta}^{\alpha[j]} g\big](x)|^2, \quad \bdeta \in P_j(\bmu_j),\, x \in Q_R,
\end{equation}
where the $a_{j,\bdeta}^{\alpha[j]} \in C^{\infty}(U_{j,R})$ satisfy $|a_{j,\bdeta}^{\alpha[j]}(s)| \leq 1$ and $b_{\alpha[j]} := \prod_{i=1}^n \tfrac{C^{\alpha[j]_i}}{\alpha[j]_i!}$ for an admissible constant $C \geq 1$.\par 

For each $1\leq j\leq k$, define the Hilbert space $H_j:=L^2(P_j(\bmu_j))$ and consider the $H_j$-valued operator $\bE_{S_{j,R}}$ given by the mapping
\begin{equation*}
    \bE_{S_{j,R}}\bg_j(x):=(E_{S_{j,R}}g_{j,\bdeta}(x))_{\bdeta}, \qquad  \bdeta\in P_j(\bmu_j),\, x\in \R^n,
\end{equation*}
where $\bg_j:=(g_{j,\bdeta})_{\bdeta}\in L^2(U_{j,R};H_j)$. Observe that, by Fubini's theorem, for any $x\in \R^n$ and $h_j\in H_j$, we have 
\begin{equation*}
  \langle \bE_{S_{j,R}}\bg_j(x),h_j\rangle_{H_j}=E_{S_{j,R}}(\langle\bg_j(\cdot),h_j\rangle_{H_j})(x).  
\end{equation*}
Hence $\bE_{S_{j,R}}$ coincides with the $H_j$-valued extension of $E_{S_{j,R}}$, as described in the appendix. Letting $\kappa := \#\cJ^{\star}$, we write $\balpha = (\alpha[j])_{j \in \cJ^{\star}}$ for $\balpha \in (\N_0^n)^{\kappa}$, $\bf_j^{\balpha}:=(f_{j,\bdeta}^{\balpha})_{\bdeta}$ with $f_{j, \bdeta}^{\balpha} := a_{j,\bdeta}^{\alpha[j]} f_{j, \bdeta}$. Using this notation, and applying \eqref{eq: applying slice}, we have
\begin{align}
    I(R;\bmu) & \lesssim \int_{Q_R} \prod_{j  \in \cJ_R}  |E_{S_{j,R}}f_j(x)|^{q_j} \prod_{j \in \cJ^{\star}} \Big( \sum_{\alpha[j] \in \N_0^n} b_{\alpha[j]} \| \bE_{S_{j,R}}\bf_{j}^{\balpha}(x) \|_{H_j}^2  \Big)^{q_j/2} \,\ud x \notag \\
    &\lesssim \sum_{\balpha \in (\N_0^n)^{\kappa}} \prod_{j \in \cJ^{\star}} b_{\alpha[j]}^{q_j/2}  \int_{Q_R} \prod_{j  \in \cJ_R}  |E_{S_{j,R}}f_j(x)|^{q_j} \prod_{j \in \cJ^{\star}} \|\bE_{S_{j,R}}\bf_{j}^{\balpha}(x)\|_{H_j}^{q_j}\,\ud x. \label{eq: uncer 2}
\end{align}
Here the last inequality follows by the nesting of the  $\ell^{2/q_j}$ and $\ell^1$ norms for $2/q_j \geq 1$.

If $j \in \cJ^{\star}$, then, provided $\rho_{\circ} > 0$ is chosen sufficiently small, we have $\rho(E_{S_{j,R}}) < \rho_{\circ}^{-1} \rho(E_{S_j^{\star}})$. Our hypotheses therefore ensure
\begin{equation*}
    \rho(E_{S_{j,R}}) < \rho_{\circ}^{-1} \rho_{\circ}^{n - \ell^{\star}(j)} \leq \rho_{\circ}^{n - \ell_{R, \btmu}(j)} \qquad \textrm{for $j \in \cJ^{\star}$,}
\end{equation*}
since $0 < \rho_{\circ} < 1$ and $\ell_{R,  \btmu}(j) \geq \ell^{\star}(j) + 1$. Note also that $\cL_{R,  \btmu}(j) \subseteq \cL^{\star}(j)$ and that, by \eqref{eq: support slice 2}, the functions $f_{j,\bdeta}$ are in $L^2(S_{j,R})$ for almost every $\bdeta \in P_j(\bmu_j)$ and are supported in $\cN_{ \tilde{\mu}_{j,\ell}} S_{j,\ell}$ for all $\ell \in \cL_{R, \btmu}(j)$.

In light of the preceding discussion, we can apply the inequality \eqref{eq: desired} featured in the definition of $\fR(R; \btmu)$ to the extension operators $E_{S_{j,R}}$ and the functions $(f_j)_{j \in \cJ_R}$, $(f_{j,\bdeta_j}^{\balpha})_{j \in \cJ^{\star}}$ as defined above. Moreover, the inequality \eqref{eq: desired} can be lifted to a vector-valued variant using a multilinear extension of the Marcinkiewicz--Zygmund theorem (see Proposition \ref{prop: multiMZ}). Combining these observations,
\begin{align}\label{eq: uncer 3}
    \int_{Q_R}& \prod_{j  \in \cJ_R}  |E_{S_{j,R}}f_j(x)|^{q_j}  \prod_{j  \in \cJ^{\star}} \|\bE_{S_{j,R}}\bf_{j}^{\balpha}(x)\|_{H_j}^{q_j}\,\ud x \\
    \nonumber
   & \lesssim \fR(R;  \btmu)  \prod_{j = 1}^k \prod_{\ell \in \cL_{R, \btmu}(j)} (\mu_{j,\ell})^{m(j,\ell)q_j/2} \prod_{j  \in \cJ_R} \|f_j\|_2^{q_j} \prod_{j  \in \cJ^{\star}} \|\|\bf_j^{\balpha}\|_{H_j}\|_2^{q_j}.
\end{align}
Here we have used the fact that $\tilde{\mu}_{j,\ell} \leq C_{\circ}^2 \mu_{j, \ell}$ for all $1 \leq j \leq k$ and $\ell \in \cL_{R, \btmu}(j)$.

By the pointwise bound for the $a_{j,\bdeta}^{\alpha[j]}$ and a change of variables, 
\begin{equation*}
   \|\|\bf_j^{\balpha}\|_{H_j}\|_2 =  \Big( \int_{P_j(\bmu_j)} \|f_{j, \bdeta_j}^{\balpha}\|_2^2\,\ud \bdeta_j \Big)^{1/2} \lesssim \|f_j\|_2 \qquad \textrm{for $j \in \cJ^{\star}$.}
\end{equation*}
Combining this with \eqref{eq: uncer 3}, we obtain
\begin{align}\label{eq: uncer 4}
    \int_{Q_R} \prod_{j  \in \cJ_R} &|E_{S_{j,R}}f_j(x)|^{q_j}  \prod_{j \in \cJ^{\star}} \|\bE_{S_{j,R}}\bf_{j}^{\balpha}(x) \|_{H_j}^{q_j}\,\ud x \\
    \nonumber
    &\lesssim \fR(R; \btmu) \prod_{j = 1}^k \prod_{\ell \in \cL_{R, \btmu}(j)} (\mu_{j,\ell})^{m(j,\ell)q_j/2} \prod_{j  = 1}^k \|f_j\|_2^{q_j}.
\end{align}

We apply \eqref{eq: uncer 4} to bound the right-hand side of \eqref{eq: uncer 2}, noting that the sequence $(b_{\alpha}^{q_j/2})_{\alpha \in \N_0^n}$ is (absolutely) summable. Combining the resulting inequality with \eqref{eq: uncer 1}, we deduce the desired bound. 
\end{proof}




\subsection{Recursive step}\label{subsec: recursive step} Let $p_j := q_j/2$ and $1 \leq j \leq k$. For $R \geq 1$, let $\fK(R)$ be the smallest constant $C \geq 1$ for which the inequality
\begin{equation*}
    \int_{Q_R} \prod_{j = 1}^k \Big| \sum_{T_j \in \bT_j} c_{T_j} \chi_{T_j} \Big|^{p_j} \leq CR^{n/2} \prod_{j=1}^k \Big( \sum_{T_j \in \bT_j} |c_{T_j}| \Big)^{p_j}
\end{equation*}
holds whenever, for each $1 \leq j \leq k$, there exists some $1 \leq \ell_j \leq r_j$ such that $\bT_j$ is a countable set of slabs of the form
    \begin{equation}\label{eq: tubes for Kakeya}
        T_j = \big\{ x \in \R^n : \big| (\partial_s \Sigma_{j,\ell_j})(s)^{\top}x - v\big| < R^{1/2} \big\}
    \end{equation}
    where $s$, $v \in \R^{d_j'}$ with $|s|_{\infty} < \rho_{\circ}$ for $d_j' := \dim(S_{j,\ell_j})$ and $(c_{T_j})_{T_j \in \bT_j} \in \ell^1(\bT_j)$.

Let $\bmu = (\bmu_j)_{j = 1}^k$ be a compatible family of scales for $\sE$, with $\bmu_j = (\mu_{j,\ell})_{\ell = 1}^{r_j}$ and $0 < \mu_{j,\ell} \leq \mu_{\circ}$ for $1 \leq \ell \leq r_j$, $1 \leq j \leq k$. Recall the constant $C_\circ \geq 1$ and the definition of $\btmu$ from \S\ref{subsec: uncertainty red}. Let $\bttmu$ denote the compatible family of scales which result from applying the operation $\bmu \mapsto \btmu$ twice to $\bmu$. Provided $C_{\circ} \geq 1$ is chosen sufficiently large, the following holds.

\begin{lemma}\label{lem: recursive step} For all $\varepsilon > 0$, $R \geq C_{\circ}^8$ and $\bmu$ compatible for $\sE$, we have 
\begin{equation*}
    \fR(R; \bmu) \lesssim_{\varepsilon} R^{\varepsilon} \fR(R^{1/2};  \bttmu)\fK(R). 
\end{equation*}
\end{lemma}

Once Lemma~\ref{lem: recursive step} is established, Theorem~\ref{thm: nested} may be deduced as follows. 

\begin{proof}[Proof (of Theorem~\ref{thm: nested})] We work with the fixed ensemble $\sE$ introduced at the beginning of the section. Without loss of generality, we may assume our compatible family of scales $\bmu = (\bmu_j)_{j = 1}^k$ with $\bmu_j = (\mu_{j,\ell})_{\ell = 1}^{r_j}$ satisfies $0 < \mu_{j,\ell} \leq \mu_{\circ}$ for $1 \leq \ell \leq r_j$, $1 \leq j \leq k$. Indeed, otherwise we can pass to a (strict) subensemble $\mathscr{F}$ of $\sE$ (consisting precisely of those $S_{j,\ell}$ for which $0 < \mu_{j,\ell} \leq \mu_{\circ}$ holds) and the analysis is simpler. This reduction can be formalised using an induction on the number of submanifolds in the ensemble.

In view of \eqref{eq: poly size}, we can assume $R$ is sufficiently large, namely $R \geq C_{\circ}^8$. By the hypothesis $\BLreg{\bL(\mathscr{E})}{\bp}< \infty$ and  Proposition~\ref{prop: BL ensemble}, we have that $\BLreg{\bL(\mathscr{F})}{\bp}<\infty$ for all subensembles $\mathscr{F}$ of $\mathscr{E}$. For $\varepsilon>0$, let $\nu_{\mathscr{F}}>0$ and $C_{\mathscr{F},\varepsilon}\geq 1$ be the constants appearing in Theorem \ref{thm: Mald} for each datum $(\bL(\mathscr{F}),\bp)$, and set
\begin{align*}
\overline{\nu}_{\mathscr{E}}&:=\min\{\nu_{\mathscr{F}} : \mathscr{F} \text{ subensemble of } \mathscr{E} \}>0   \\
\overline{C}_{\mathscr{E},\varepsilon}&:=\max \{C_{\mathscr{F},\varepsilon} : \mathscr{F} \text{ subensemble of } \mathscr{E}\}< \infty.
\end{align*}
Provided $\rho_{\circ} > 0$ is sufficiently small, the slabs \eqref{eq: tubes for Kakeya} have core planes which are, modulo translations, within distance $\overline{\nu}_{\mathscr{E}}$ of the fixed subspaces  $\ker \partial_s \Sigma_{j,\ell_j}(0)^\top$. Since,  by \eqref{eq: Phi vs gamma}, any  $(\partial_s \Sigma_{j,\ell_j}(0)^\top)_{j=1}^k =( (\partial_s \sigma_{\ell_j}(0))^\top (\partial_s \Sigma_j (0))^\top )_{j=1}^k$ corresponds to $\bL(\mathscr{F})$ for some subensemble $\mathscr{F}$ of $\mathscr{E}$, it then follows from Theorem~\ref{thm: Mald}, taking $\lambda=R^{1/2}$, that for every $\varepsilon>0$, we have $\fK(R)\leq \overline{C}_{\sE,\varepsilon} R^\varepsilon$ for all $R \geq 1$. Combining this with Lemma~\ref{lem: recursive step}, for all $\varepsilon > 0$ we have 
\begin{equation}\label{eq: recursive step}
    \fR(R; \bmu) \lesssim_{\varepsilon} R^{\varepsilon} \fR(R^{1/2};  \bttmu) \qquad \textrm{for all $R \geq C_\circ^8$. }
\end{equation}

Recall the definition of the upper exponent $\eta_{\mathrm{exp}}$ from \eqref{eq: upper exponent} which, by \eqref{eq: poly size}, is a well-defined real number. The inequality \eqref{eq: recursive step} implies that $\eta_{\mathrm{exp}} \leq \eta_{\mathrm{exp}}/2$ and so $\eta_{\mathrm{exp}} \leq 0$. Thus, for all $\varepsilon >0$, we have $\fR(R) \lesssim_{\varepsilon} R^{\varepsilon}$. Let $\cL^{\star}(j) = \{1, \dots, r_j\}$, so that $\ell^{\star}(j) = 0$, for $1 \leq j \leq k$. Applying Lemma \ref{lem: uncer red} with these $\cL^{\star}(j)$ together with the above estimate for $\fR(R)$, the desired bound follows with $\rho_{\sE} := \rho_{\circ}^n$. 
\end{proof}

We turn to the proof of Lemma \ref{lem: recursive step}, which is a variant of an argument from~\cite{BCT2006, BBFL2018}.

\begin{proof}[Proof (of Lemma \ref{lem: recursive step})] Let $\varepsilon > 0$ and, recalling the definitions from \S\ref{subsec: wp}, set 
\begin{equation*}
    \bT_j := \bT_{R,\bmu}(\cS_{j,R}), \qquad 1 \leq j \leq k.
\end{equation*}
Using the notation in \eqref{eq: recursive manifolds}, the slabs in $\bT_j$ have width $R^{1/2+\varepsilon_{\circ}}$, where $\varepsilon_{\circ} := \varepsilon/(100n)$, and their core planes are normal to the tangent plane to $S_{j, R^{1/2}} = \Sigma_{j, R^{1/2}}(U_{j,R^{1/2}})$ at $\Sigma_{j, R^{1/2}}(s)$ for some $s \in U_{j,R^{1/2}}$: that is, they are of the type \eqref{eq: tubes for Kakeya}. Provided $\rho_{\circ} > 0$ is sufficiently small,
\begin{equation*}
   |s|_\infty \leq \rho_\circ^{-1}\rho(E_{S_{j,R}}) < \rho_{\circ}^{n-\ell_{R,\bmu}(j)-1} \leq \rho_\circ, \qquad \textrm{since $\ell_{R,\bmu}(j) +1 \leq r \leq n-1$.}
\end{equation*}
 This observation will allow for an application of the bound $\fK(R)$.

Let $\cQ_{R^{1/2}}$ denote a cover of $Q_R := [-R,R]^n$ by essentially disjoint cubes of sidelength $R^{1/2}$, so that 
\begin{equation}\label{eq: itn 0}
    \int_{Q_R} \prod_{j = 1}^k |E_{S_{j, R}}f_j|^{q_j} \leq \sum_{Q \in \cQ_{R^{1/2}}} \int_{Q} \prod_{j = 1}^k |E_{S_{j, R}}f_j|^{q_j}.
\end{equation}
Fix $Q \in \cQ_{R^{1/2}}$. Provided $\rho_\circ>0$ is chosen sufficiently small, we can perform the wave packet decomposition from \S\ref{subsec: wp}. By the spatial localisation property,
\begin{equation*}
    |E_{S_{j, R}}f_j(x)| \lesssim_{N,\varepsilon} |E_{S_{j, R}}f_{j, Q}(x)| + R^{-N} \|f_j\|_2 \qquad \textrm{for all $N \in \N_0$ and $x \in Q$,}
\end{equation*}
where
\begin{equation*}
    f_{j, Q} := \sum_{T_j \in \bT_j[Q]} f_{j, T} \qquad \textrm{for} \qquad \bT_j[Q] := \{ T \in \bT_j : T \cap Q \neq \emptyset \}.
\end{equation*}
Thus,
\begin{equation*}
    \int_Q \prod_{j = 1}^k |E_{S_{j, R}}f_j|^{q_j} \lesssim_{\varepsilon} \int_Q \prod_{j = 1}^k |E_{S_{j, R}}f_{j, Q}|^{q_j} + R^{-100n} \prod_{j = 1}^k \|f_j\|_2^{q_j}.
\end{equation*}

For each $T\in\bT_j$, we have $\supp f_{j,T_j}\subseteq B(0,\rho_\circ+4\cdot R^{-1/2})$. Thus, for large $R$, we can guarantee that $\supp f_{j,Q} \subseteq W_0^{(j)}$. By the preservation of support property \eqref{eq: preserve supp}, provided $C_{\circ} \geq 1$ is chosen sufficiently large, each $f_{j,Q}$ satisfies $\supp f_{j,Q} \subseteq \cN_{C_{\circ}^2\bmu_j} \cS_j$. Moreover, by the support properties of our amplitude functions, we may assume without loss of generality that $\supp f_{j,Q} \subseteq \cN_{\btmu_j} \cS_j$ for $\btmu_j$ as defined in \S\ref{subsec: uncertainty red}. 

Define $\cL^{\star}(j) := \cL_{R, \bmu}(j)$ and, since $\mu_{j,\ell} \geq  C_{\circ}^{-4}\dbtilde{\mu}_{j,\ell} \geq R^{-1/2}\dbtilde{\mu}_{j,\ell}$, note that
\begin{equation*}
    \cL_{R^{1/2}, \bttmu}(j) \subseteq \cL^{\star}(j) \subseteq \{1, \dots, r_j\} \qquad \textrm{for $1 \leq j \leq k$.}
\end{equation*} 
From the observations of the previous paragraph, the $f_{j, Q}$ satisfy the hypothesis \eqref{eq: uncer red supp} of Lemma~\ref{lem: uncer red} with $\bmu$ replaced with $\btmu$. We may therefore apply Lemma \ref{lem: uncer red} to deduce that
\begin{equation}\label{eq: itn 1}
    \int_Q \prod_{j = 1}^k |E_{S_{j, R}}f_{j, Q}|^{q_j}  \lesssim  \fR(R^{1/2};  \bttmu) \bC_R(\bmu) \prod_{j = 1}^k  \|f_{j,Q}\|_2^{q_j},
\end{equation}
where $\bC_R(\bmu)$ is as defined in \eqref{eq: recursive const}.\footnote{Observe that, for our choice of $\cL^{\star}(j)$ and $\bmu$ replaced with $\btmu$, the factors on the right-hand side of \eqref{eq: uncer red} satisfy
\begin{equation*}
    \prod_{\ell \in \cL^{\star}(j)} (\tilde{\mu}_{j,\ell})^{m(j,\ell) q_j/2} \lesssim \prod_{\ell \in \cL_{R, \bmu}(j)} (\mu_{j,\ell})^{m(j,\ell) q_j/2}, 
\end{equation*}
justifying the appearance of $\bC_R(\bmu)$ in \eqref{eq: itn 1}.} Here we have also used the translation invariance of the estimates to recentre $Q$ at the origin. 

Taking the sum over $\cQ_{R^{1/2}}$ on both sides of \eqref{eq: itn 1} and combining the resulting estimate with \eqref{eq: itn 0}, we deduce that 
\begin{equation}
 \label{eq: itn 2}
 \int_{Q_R} \prod_{j = 1}^k |E_{S_{j, R}}f_j|^{q_j} \lesssim_{\varepsilon} \fR(R^{1/2}; \bttmu )  \bC_R(\bmu) \sum_{Q \in \cQ_{R^{1/2}}}  \prod_{j = 1}^k \|f_{j,Q}\|_2^{q_j}
\end{equation}
holds up to the inclusion of a rapidly decaying error term. It remains to estimate the right-hand sum in \eqref{eq: itn 2}.

By the orthogonality properties of the wave packets,
\begin{equation*}
    \prod_{j = 1}^k \|f_{j,Q}\|_2^2 \lesssim \prod_{j = 1}^k \Big( \sum_{T_j \in \bT_j[Q]} \|f_{j,T_j}\|_2^2\Big).
\end{equation*}
Furthermore, if $T_j \in \bT_j[Q]$, then it follows that $Q \subset T_j^{\,\circ}$, where $T_j^{\,\circ}$ denotes the slab with the same core plane as $T_j$ but with width scaled by a factor of $2$. Consequently,
\begin{equation*}
    \prod_{j = 1}^k \|f_{j,Q}\|_2^{q_j} \lesssim \frac{1}{|Q|}\int_Q \prod_{j = 1}^k \big| \sum_{T_j \in \bT_j} \|f_{j,T_j}\|_2^2 \cdot \chi_{T^{\,\circ}_j}(x)\big|^{q_j/2}\,\ud x.
\end{equation*}
Summing this inequality over all $Q \in \cQ_{R^{1/2}}$ and combining the resulting estimate with \eqref{eq: itn 2}, we obtain 
\begin{equation}\label{eq: itn 3}
    \int_{Q_R} \prod_{j = 1}^k |E_{S_{j,R}}f_j|^{q_j} \lesssim_{\varepsilon}\fR(R^{1/2}; \bttmu) \bC_R(\bmu)  R^{-n/2}  \int_{Q_R}\prod_{j = 1}^k \big|\sum_{T_j \in \bT_j} \|f_{j,T_j}\|_2^2 \cdot \chi_{T^{\,\circ}_j}\big|^{p_j}, 
\end{equation}
up to the inclusion of a rapidly decaying error term,
where we recall $p_j := q_j/2$ for $1 \leq j \leq k$.

By the definition of the constant $\fK(R)$ and the orthogonality of the wave packets,
\begin{align}
\nonumber
    \int_{Q_R} \prod_{j = 1}^k \big|\sum_{T_j \in \bT_j} \|f_{j,T_j}\|_2^2 \cdot \chi_{T^{\,\circ}}\big|^{p_j} &\lesssim_{\varepsilon} \fK(R) R^{n/2 + \varepsilon} \prod_{j = 1}^k \big(\sum_{T_j \in \bT_j} \|f_{j,T_j}\|_2^2\big)^{p_j} \\
    \label{eq: itn 4}
    &\lesssim \fK(R) R^{n/2 + \varepsilon} \prod_{j = 1}^k \|f_j\|_2^{q_j}.
\end{align}
Note that the slabs featured in the above displays (whose geometry coincides with that of the slabs defined in \eqref{eq: slabs 1}), are slightly wider (by a factor of $R^{\varepsilon_{\circ}}$ where $\varepsilon_{\circ} := \varepsilon/(100n)$) than those appearing in the definition of $\fK(R)$. This discrepancy can be dealt with using a simple covering argument, which accounts for the additional $R^{\varepsilon}$ factor on the right-hand side of the above inequality. 

Combining \eqref{eq: itn 4} with \eqref{eq: itn 3}, we obtain
\begin{equation*}
    \int_{Q_R} \prod_{j = 1}^k |E_{S_{j,R}}f_j|^{q_j} \lesssim_{\varepsilon} R^{\varepsilon}\fR(R^{1/2}; \bttmu)\fK(R)  \bC_R(\bmu) \prod_{j = 1}^k \|f_j\|_2^{q_j}
\end{equation*}
and therefore, by definition, $\fR(R; \bmu) \lesssim_{\varepsilon} R^{\varepsilon} \fR(R^{1/2}; \bttmu) \fK(R)$, as required. 
\end{proof}




\appendix




\section{A multilinear Marcinkiewicz--Zygmund theorem} 

Let $H$ be a separable complex Hilbert space and $\{e_{a}\}_{a \in \cA}$ be a choice of orthonormal basis for $H$, where $\cA$ is a countable set. Given any $g \in H$, we may write
\begin{equation*}
    g = \sum_{a \in \cA} \inn{g}{e_{a}} \quad \textrm{where} \quad \|g\|_{H} = \big(\sum_{a \in \cA} |\inn{g}{e_{a}}|^2\big)^{1/2}.
\end{equation*}
Let $(\Omega, \Sigma,\mu)$ and $(\Omega',\Sigma',\mu')$ be measure spaces. We say that $\bf \colon \Omega' \to H$ is $\Sigma'$-measurable if $\inn{\bf}{g}$ is a $\Sigma'$-measurable function for all $g \in H$. 
Let $T$ be a linear mapping sending $\Sigma'$-measurable functions on $\Omega'$ to $\Sigma$-measurable functions on $\Omega$. We define \textit{the vector-valued extension} $\bT$ to be the operator mapping $H$-valued measurable functions on $\Omega'$ to $H$-valued measurable functions on $\Omega$ given by
\begin{equation*}
    \inn{\bT\bf(x)}{e_{a}} := T (\inn{\bf}{e_{a}})(x) \qquad \textrm{for all $\bf \colon \Omega' \to H$ measurable and all $a \in \cA$.}
\end{equation*}
This uniquely defines the operator $\bT$, which is linear.

The classical Marcinkiewicz--Zygmund inequality \cite{Marcinkiewicz} for separable Hilbert spaces has the following multilinear counterpart. 

\begin{proposition}[Multilinear Marcinkiewicz--Zygmund Theorem]\label{prop: multiMZ}
Let $(\Omega, \Sigma, \mu)$ and $(\Omega_j, \Sigma_j, \mu_j)$ for $1 \leq j \leq m$ be $\sigma$-finite measure spaces. For $1 \leq j \leq m$, let $T_j$ be linear mappings sending $\Sigma_j$-measurable functions on $\Omega_j$ to $\Sigma$-measurable functions on $\Omega$. Let and $0 < p_j, q_j < \infty$ for $1 \leq j \leq k$ and further suppose that there exists some $M \geq 0$ such that
\begin{equation*}
    \int_{\Omega} \prod_{j=1}^k |T_jf_j(x)|^{q_j} \,\ud \mu(x) \leq M \prod_{j=1}^k \|f_j\|_{L^{p_j}(\Omega_j, \mu_j)}^{q_j}
\end{equation*}
holds for all $\Sigma_j$-measurable functions $f_j$, $1 \leq j \leq k$.

Let $H_j$ be separable complex Hilbert spaces and let $\bT_j$ denote the vector-valued extensions of $T_j$, $1 \leq j \leq k$. Then the inequality
\begin{equation*}
    \int_{\Omega}\prod_{j=1}^k \|\bT_j\bf_j(x)\|_{H_j}^{q_j} \,\ud \mu(x) \lesssim_{p,q} M \prod_{j=1}^k \|\|\bf_j\|_{H_j}\|_{L^{p_j}(\Omega_j, \mu_j)}^{q_j}
\end{equation*}
holds for all $\Sigma_j$-measurable functions $\bf_j \colon \Omega_j \to H_j$, $1 \leq j \leq k$.
\end{proposition}

If the Hilbert spaces $H_j$ are finite dimensional, the result follows from a standard Khintchine's inequality argument and the Fubini--Tonelli theorem. The extension to separable spaces follows from Fatou's lemma. We omit the details.




\bibliography{Reference}
\bibliographystyle{amsplain}

\end{document}